\documentclass[reqno,11pt]{amsart}
\usepackage[utf8]{inputenc}
\DeclareUnicodeCharacter{0301}{}

\calclayout
\usepackage{amsthm,amsfonts,amstext,amssymb,mathrsfs,amsmath,latexsym,mathtools,mathabx}
\usepackage{enumerate}
\usepackage{bm,accents}
\usepackage[abs]{overpic}
\usepackage{mathrsfs}
\usepackage{hyperref}
\usepackage{esint} 

\theoremstyle{plain}
\newtheorem{thm}{Theorem}[section]
 
\newtheorem{prop}[thm]{Proposition}

\newtheorem{cor}[thm]{Corollary} 

\theoremstyle{definition}

\theoremstyle{remark}
\newtheorem{remark}[thm]{Remark}
\numberwithin{equation}{section}

\numberwithin{equation}{section}

\renewcommand{\Re}{\mathop{\rm Re}}
\renewcommand{\Im}{\mathop{\rm Im}}

\DeclareMathOperator{\supp}{supp}

\DeclareMathOperator{\conv}{conv}

\newcommand{\N}{\mathbb{N}}
\newcommand{\C}{\mathbb{C}}
\newcommand{\D}{\mathbb{D}}
\newcommand{\R}{\mathbb{R}}

\newcommand{\isdef}{\overset{\mathrm{def}}{=\joinrel=}}

\title[Flow of the zeros of derivatives]{Flow of the zeros of polynomials under iterated differentiation }

\author[A.~Mart\'{\i}nez-Finkelshtein]{Andrei Mart\'{\i}nez-Finkelshtein}

\address[AMF]{Department of Mathematics, Baylor University, Waco, TX 76706, USA, and Department of Mathematics, University of Almer\'{\i}a, Almer\'{\i}a, Spain}

\email{A\_Martinez-Finkelshtein@baylor.edu}

\author[E.~A.~Rakhmanov]{Evguenii~A.~Rakhmanov}

\address[ER]{Department of Mathematics, University of South Florida, Tampa, FL 33620, USA}

\email{rakhmano@usf.edu}

\date{\today}



\begin{document}

\begin{abstract}
For a monic polynomial $Q_n$ of degree $n$, let $Q_{n, k}$ be its $k$-th derivative normalized to be monic. Under the only assumption that the sequence $\{Q_n\}$ has a weak* limiting zero distribution (an empirical distribution of zeros) represented by a probability measure $\mu_0$ with compact support in the complex plane, we show that as $n, k \rightarrow \infty$ such that $k / n \rightarrow t \in(0,1)$, the Cauchy transform of the normalized zero-counting measure of the polynomials $Q_{n, k}$ converges in a neighborhood of infinity to an analytic function, uniquely determined by $\mu_0$ and $t$, that can be written as the Cauchy transform of a measure $\mu_t$, not necessarily uniquely determined unless $\mu_0$ is supported on the real line.

The family of these Cauchy transforms and, when well defined, the corresponding measures $\mu_t $, $t \in(0,1)$, whose dependence on the parameter $t$ can be interpreted as a flow of the zeros under iterated differentiation, has several interesting connections with the inviscid Burgers equation, the fractional free convolution of $\mu_0$, or a nonlocal diffusion equation governing the density of $\mu_t$ on $\R$. 

We provide an elementary and unified approach that not only recovers, but also explains various phenomena observed in prior works -- from Burgers-type PDEs to free probability limits.
\end{abstract}

\maketitle

\section{Introduction} \label{sec:intro}

In recent years, the problem of repeated differentiation of polynomials has been intensively studied from different points of view and in different connections. In this Introduction, we start by formulating our main result. After that, we make a few remarks about the classical roots of the problem and briefly comment on some recent results.

Hereafter, the asymptotic zero distribution of a sequence of polynomials (also known as its empirical distribution of zeros) is understood in terms of the weak-* limits of their zero-counting measures. We define the zero-counting measure  of a polynomial $P$ as 
\begin{equation} \label{eq:countingmeasure}
	\chi(P)\isdef \sum_{P(x)=0} \delta_x,
\end{equation}
where the sum is taken over all zeros of $P$ with account of multiplicities, and $\delta_x$ is the Dirac measure (unit mass) supported at $x$. This way,  $\chi(P)$ is independent of  normalization of $P$, and  
$$
\chi(P)(\mathbb{C})=\deg(P).
$$ 
We are interested in the following problem: given a sequence of (monic) polynomials $Q_n$, $\deg Q_n=n$, $n\in \N$, and the triangular table of their iterated derivatives 
\begin{equation}
    \label{def:QN}
Q_{n,k}(x) \isdef  \frac{(n-k)!}{n!}\, \frac{d^k}{dx^k } Q_n(x) = x^{n-k}+\text{lower degree terms}, \quad k=0, 1, \dots, n-1,
\end{equation}
what are the possible limit distributions of the normalized zero-counting measures
\begin{equation} \label{eq:defSigmaNK}
\sigma_{n,k} \isdef  \frac{1}{ n} \chi\left(Q_{n, k }\right)  
\end{equation}
as $n\to \infty$ in such a way that $k/n\to t\in [0,1)$, if we only know that $\sigma_{n,0}$ have a weak-* limit, compactly supported on $\C$?

A weaker form of this question deals with the convergence of the logarithmic derivatives
\begin{equation} \label{eqCauchyTransf00}
v_{n,k} (z)\isdef	 \frac{  Q'_{n,k}(z)}{n\, Q_{n,k}(z)}  .
\end{equation}
The main result of this paper is the following:
\begin{thm} \label{mainthm0}
 Let $Q_n$ be a sequence of polynomials whose zeros belong to a compact convex set $S$.  Assume also that the limit 
 $$
 \lim_n  \frac{  Q'_{n}(z)}{n\, Q_{n}(z)} = u(z)
 $$
exists uniformly on compact subsets of $\Omega= \C\setminus S$. 

Then for any $t \in[0,1)$ and any sequence $k_n$ of natural numbers with $k_n / n \rightarrow t$ we have that
	\begin{equation}
	    \label{limitLogDervs}
	\lim_n  v_{n,k} (z)=\lim_n 	 \frac{  Q'_{n,k}(z)}{n\, Q_{n,k}(z)} =u(z,t)
	\end{equation}
uniformly on compact subsets of $\Omega$, where $u(z,t)$ is uniquely determined by the equation 
\begin{equation}
    \label{inviscidBurgers}
\frac{\partial u}{\partial t}(z, t)=\frac{1}{u(z, t)} \frac{\partial u}{\partial z}(z, t)
\end{equation}
with the initial condition $u(z,0)=u(z)$. 
\end{thm}
Equation \eqref{inviscidBurgers} is a special case of a nonlinear wave equation, known as the inviscid Burgers or Hopf equation.\footnote{\, It goes by various other names such as dispersionless Burgers' equation,  dispersionless KdV (Korteweg-de Vries) equation, inviscid KdV, see, e.g.~\cite[Section 2.1.3]{Arendt2023}, \cite[Section 3.4]{MR1625845} or \cite[Ch. 9]{MR2309862}.} It is traditionally considered for the real values of the variables $z$ and $t$, but it has also been studied on the complex plane, both in its viscous and inviscid form; see, for example, \cite{MR3573689, kabluchko2023fractional, kabluchko2023heat, MR3773856, MR4259446, VANDENHEUVEL2023133686}. 

The Hopf equation is known for a remarkable phenomenon: its solutions can develop singularities (or shocks), even when posed on the real line. However, this behavior does not occur when the initial condition $u(z,0)$ is given by the Cauchy transform of a unit \textit{positive} measure supported on a bounded convex set $S$. In this case, the solution $u(z,t)$ to the equation \eqref{inviscidBurgers} remains analytic in the domain $\Omega = \C \setminus S$.
	
Moreover, for any fixed $t \in (0,1)$, the solution $u(z,t)$ can itself be represented as the Cauchy transform of a positive measure supported on $S$---although this measure is not necessarily unique. This apparently new observation is implicitly contained in the statement of Theorem~\ref{mainthm0}. Indeed, the initial measure can be approximated by the zero-counting measures of a sequence of polynomials $Q_n$, whose zeros are all contained in $S$. The analyticity of $u(z,t)$ in $\Omega$ then follows from the convergence in \eqref{limitLogDervs}. The representation of $u$ as a Cauchy transform follows from the weak compactness of the associated sequence of zero-counting measures.
	
It is important to note that the analyticity statement fails if the assumption of positivity of $\mu_0$ is dropped; see Remark~\ref{remark22} in Section~\ref{sec:main}.

The fact that in the situation of complex zeros, not much more than \eqref{limitLogDervs} can be said is illustrated by the paradigmatic example of two sequences of polynomials, $Q_n(z)=z^n$ and $Q_n(z)=z^n-1$. In both cases, the sequence of normalized zero-counting measures $\sigma_{n,k_n}$, defined in \eqref{eq:defSigmaNK}, converges as $k_n/n\to t$ to the same measure $(1-t)\delta_0$, although $\mu_0$'s are clearly different: $\delta_0$ in the former case, and the unit Lebesgue measure on the unit circle, in the latter. In the notation of Theorem~\ref{mainthm0}, in the domain $0<|z|<1$, $u(z)=0$ for $z^n$, and $u(z)=1/z$ for $z^n-1$, which shows that $u(z)$ does not determine $u(z,t)$ there.

The situation is different if $Q_n$ have only real zeros or, more generally, lie on a segment of a straight line. It will be shown that  $u(z,t)$ is the Cauchy transform of a unique positive measure compactly supported on the same segment. In this context, we denote this measure by $\mu_t$. It is completely determined by the values of $u(z,t)$ in a neighborhood of infinity, and measures $\sigma_{n,k}$ weakly converge to $\mu_t$. In more general cases, when uniqueness is not guaranteed, we use $\mu_t$ to denote one of the measures whose Cauchy transform coincides with $u(z,t)$ in a neighborhood of infinity; see for example, the statement of \textit{(\ref{item3Mainthm})} in Theorem~\ref{mainthm1}.

\medskip

Let us comment on the classical roots of the problem of zeros of derivatives of polynomials.

A standard example of a flow generated by repeated differentiation is provided by the sequence of polynomials
\begin{equation} \label{eq:Rodrigues}
	Q_{2n,k} (z)=\frac{(2n-k)!}{(2n)!}\frac{d^k}{d x^k}\left(x^2-1\right)^n, \quad n=0,1,2,\dots, \quad k=[2 t n], 
\end{equation}
for some $t \in[0,1)$ (we denote by $[\cdot]$ the integer part of the argument). The notion of ``flow'' is justified by considering the parameter $t$ as ``time''.
The corresponding  normalized sequence of zero-counting measures
$\frac{1}{2 n} \chi\left(Q_{2n, [2tn] }\right) $  converges in the weak-* sense\footnote{\, The weak-* convergence in this context is a classical result and may be regarded as part of the folklore in the analytic theory of polynomials. For explicit formulas in this setting, see, for example, \cite{ShapiroRodrigues, HoskinsKabluchko21} or the derivation provided in Section~\ref{sec:examples}.} to a measure $\mu_t$ with the total mass $1-t$, which has an explicit representation:
\begin{equation}
    \label{example1a}
\mu_t = \max \left\{ \frac{1}{2}-t, 0\right\} \left(\delta_{-1}+\delta_1\right)+  \mu_t^{ac} , 
\end{equation}
where $\mu_t^{ac} $ is an absolutely continuous measure supported on the interval $[-2 \sqrt{t(1-t)}, 2 \sqrt{t(1-t)}]$, with
\begin{equation}
    \label{example1b}
d\mu_t^{ac}  (x)= \frac{1}{2 \pi} \frac{\sqrt{4 t(1-t)-x^2}}{1-x^2}\, d x. 
\end{equation}
As $t$ progresses, $\mu_t$ represents the flow of zeros of iterated derivatives, starting from its initial distribution $\mu_0=(\delta(-1)+\delta(1))/2$. 
The measure $\mu_t$ is well known: it is related to zeros of Gegenbauer polynomials and hence to the angular part of the hyperspherical harmonics (see, e.g.~\cite[Examples IV.1.18 and IV.6.2]{SaTo}). In fact, for $k=n$, \eqref{eq:Rodrigues} is just the Rodrigues formula for the Legendre polynomials, orthogonal in $[-1,1]$ with respect to weight $w(x)=1$. This also explains why $\mu_t$, for $t=1/2$, is one-half of the equilibrium measure (or the Robin distribution) of the interval $[-1,1]$. Notice also the curious symmetry of $\mu_t^{ac} $ with respect to the time reversal ($t\mapsto 1-t$).  Since every classical family of orthogonal polynomials satisfies a Rodrigues formula, each of these families originates its own flow of measures with interesting properties.
 
The formula \eqref{eq:Rodrigues} can be extended in several directions. For instance, the Type II  Hermite-Pad\'e (or multiple orthogonal) polynomials $Q_{2 n}$ of degree $2n$, defined up to normalization, by the orthogonality conditions
 \begin{equation} \label{eq:HP}
 	\int_{-1}^0 Q_{2 n}(x) x^k d x=  \int_0^1 Q_{2 n}(x) x^k d x=0, \quad k=0,1, \dots n-1,
 \end{equation}
 have the representation
 \begin{equation}
     \label{Kalyaginexample}
 Q_{2 n}(x)= \frac{(2n)!}{(3n)!}\frac{d^n}{d x^n}\left(x^3-x\right)^n,
  \end{equation}
obtained originally by Kalyagin \cite{kaliaguine:1981}, who used it to study the asymptotic properties of $Q_{2 n}$. As before, we can consider 
$$
Q_{2n,k} (z)=\frac{(3n-k)!}{(3n)!}\frac{d^k}{d x^k}\left(x^3-x\right)^n, \quad n=0,1,2,\dots, \quad k=[3 t n], 
$$
and the corresponding family of measures $\mu_t$ of their limiting zero distributions. In this case, the explicit formulas for $\mu_t$ are not as simple as before. 
Similar Rodrigues-type formulas exist for other families of semi-classical and multiple orthogonal polynomials \cite{Aptekarev:97, MR4582606, MR4582563}. Recently, the asymptotics of polynomials given by \eqref{eq:Rodrigues}, where $x^2-1$ on the right-hand side is replaced by an arbitrary polynomial $P$, possibly with complex zeros, has been considered in \cite{ShapiroRodrigues}; see Section~\ref{sec:examples} for more details and comments on some results of this work.

The interest in this subject has recently been revitalized by a series of papers by Steinerberger \cite{MR4242313, MR4011508, Steiner2021}, and several other works have followed in rapid development, establishing interesting connections with free probability and the inviscid Burgers (or Hopf) partial differential equation. 

First, Steinerberger wrote in \cite{MR4011508} a non-local diffusion equation that should describe the evolution of the absolutely continuous part $\mu_t^{ac} $ of $\mu_t$. Observing the similarities with an equation in \cite{ShlyakhtenkoTao21}, he conjectured in \cite{Steiner2021} that $\mu_t$ can be written in terms of the fractional free additive convolution of the initial measure $\mu_0$. Hoskins and Kabluchko rigorously proved this \cite{HoskinsKabluchko21}; later, it was also established by Arizmendi et al. \cite{Arizmendi21}, this time using elegant arguments involving only finite free convolution of polynomials, a notion developed in a series of works of Markus and collaborators; see \cite{MR4408504}. In the case of complex random zeros, the free probability interpretation of repeated differentiation has been obtained in \cite{campbell2023fractional}. 

The paper \cite{MR4408504} also contains another interesting motivation for studying our main problem: the average characteristic polynomial of all $k\times k$ minors of an $n\times n$ matrix $A$ can be expressed in terms of the $(n-k)$-th derivative of the characteristic polynomial of $A$, see \cite[Lemma 1.17]{MR4408504} for a precise formulation.

Other works addressing several aspects of repeated differentiation (many of them, in the context of random polynomials) are \cite{angst2023sure, MR4474893, campbell2023fractional, galligo2022modeling, MR4447137, kabluchko2022repeated, MR4458083}. A related differential operator (the heat-flow operator), which is degree-preserving and contains derivatives of all orders, was studied in \cite{kabluchko2023heat}. The authors present a transport map $T_t$, defined in terms of the Cauchy (or Stieltjes) transform of $\mu_0$, such that $T_t(\mu_0)=\mu_t$. Moreover, they show that the Cauchy transform of $\mu_t$ satisfies the inviscid Burgers equation (written in a slightly different form than the one appearing in the following); see \cite[Eq. (5.4)]{kabluchko2023heat}. As they point out, the complex Burgers equation with the initial value given by a Cauchy transform has recently been used in the description of the growth of multiple SLE, see \cite{MR3573689, MR3773856, MR4259446}. It also appears in \cite{MR2262808} in the description of the evolution of the end-points of the support of the equilibrium measure in a specific external field depending on $t$ (see also \cite{MR3302630}). Also, equation (1.17) in \cite{ShlyakhtenkoTao21} can be reduced to \eqref{inviscidBurgers}. 
Another equivalent formulation of \eqref{inviscidBurgers} appears again in a recent paper \cite{kabluchko2023fractional}; the authors even address the behavior of individual zeros and conjecture (proving it in the rotationally invariant case) that the Cauchy transform of the measure, evaluated along the path of a single root, must remain constant, and the roots should move along the characteristic curves of a certain nonlinear PDE. 

In this paper, we address the problem of the macroscopic description of the zeros of polynomials under repeated differentiation, with the most general assumptions stated in Theorem~\ref{mainthm0}. However, our goal is not only to prove (and partially reprove) the results but also to show why they are true. The main feature of our approach is that it is totally elementary and uses only basic facts from complex analysis.

Our main findings can be summarized as follows. The complex Hopf equation \eqref{inviscidBurgers} plays the main role. The connection of the problem on the flow of derivatives under differentiation with this equation lies in the elementary fact that the Cauchy transform of the zero-counting measure of a polynomial can be expressed in terms of the logarithmic derivative of this polynomial. As a first step, we show that if its initial value is given by the Cauchy transform of a compactly supported probability measure $\mu_0$, then there is a universal domain, depending only on the support of $\mu_0$, where not only do the solutions of the equation remain analytic, but they can be written as the Cauchy transform of a measure for all values of the parameter $t\in (0,1)$. 

We show that the solution of the Hopf equation by the method of characteristics is precisely the expression of the fact that $\mu_t$ can be written as the fractional free convolution of $\mu_0$.

Finally, we show that the nonlocal diffusion from \cite{MR4011508} is merely a projection of the Hopf equation onto the real line via the Sohotski-Plemelj identities (or the Stieltjes--Perron inversion formula).

\section{Main results}\label{sec:main}

Let us start by setting up the following notation: for $0<R<+\infty$, denote $ \mathbb D_R =\{z\in \C:\, |z|<R\}$, $\overline{ \D_R} =\{z\in \C:\, |z|\le R\}$, and $\Omega_R =\{z\in \C:\, R<|z|<+\infty\}=\C \setminus \overline{ \D_R} $. Notice that the open domain $\Omega_R\subset \C$ does not include infinity, while its closure in $\overline{\C}$, that is, $\overline{\Omega_R}=\{z\in \C:\, R\le|z|\le+\infty\}$ does. We also denote by $\partial_z$ the differentiation operator with respect to the real or complex parameter $z$ (each time, it will be clear from the context).

The \textit{inviscid Burgers'} or  \textit{Hopf equation}, considered in this paper, is written as
\begin{equation} \label{pde}
	\partial_t u = \frac{\partial_z u}{u}, \quad u=u(z;t). 
\end{equation}
A more familiar form $\partial_t u  + u \,\partial_z u =0$ can be obtained, for instance, by replacing $u$ by $-1/u$. We consider that the parameter $t$ takes real values in $[0,1]$, while $z$ is on the complex plane; in our case, $z$ is in a neighborhood of infinity, that is, a domain of the form $\Omega_R$. 

It is well known that the solutions of \eqref{pde} can exhibit shocks in a finite time $t$; we will be interested in a very particular type of initial value problem for \eqref{pde}, when $u(z,0)$ is the Cauchy transform of a probability measure.

Given a finite Borel (in general, signed) measure $\mu$ on $\C$, we define its \textit{Cauchy} (also known as Stieltjes or Markov) \textit{transform}  as
\begin{equation} \label{def:Cauchy}
	\mathcal C^\mu(z)\isdef  \int \frac{d\mu(y)}{z-y},
\end{equation}
which is an holomorphic function in $\C \setminus \supp(\mu)$.  Moreover, if $\mu$ is compactly supported, then 
\begin{equation}
    \label{seriesCauchyTransform}
\mathcal C^\mu(z) = \sum_{k=0}^\infty \frac{m_k}{z^{k+1}} 
\end{equation}
converges in a neighborhood of infinity; the coefficients  
\begin{equation}
    \label{moments}
    m_j =\int z^j \, d\mu(z), \quad j=0, 1, 2,\dots
\end{equation}
are the \textit{moments} of the measure $\mu$. 

Let $\mu_0$ be a compactly supported Borel probability measure on $\C$, and let $0<r<+\infty$ be such that the support $\supp \mu_0 \subset \D_r$. One of the central results of this paper is the fact that the initial value problem for the equation \eqref{pde} with the initial datum
\begin{equation} \label{pdeInitial}
	u(z,0)=\mathcal C^{\mu_0}(z)
\end{equation}
has an analytic solution $u(z,t)$ for all $t\in [0, T]$, $0<T<1$, and $z$ in a neighborhood of infinity that depends only on $r$ and $T$, but not on the measure $\mu_0$ itself:
\begin{thm} \label{mainthm1}
	Let $0<T<1$ and $0<r<+\infty$ be fixed. Then 
	\begin{enumerate}[(i)]
		\item \label{item1Mainthm}there exists an $R=R(T)>2r$, such that for any probability  measure $\mu_0$ with $\supp(\mu_0) \subset \D_r$, and for $(z,t)\in \Omega_R \times [0,T] $ there exists a unique solution $u(z,t)$ of \eqref{pde} satisfying the initial condition 
		$$
		u(z, 0)= \mathcal C^{\mu_0}(z).
		$$
		\item \label{item2Mainthm} For $(z,t)\in \Omega_R \times [0,T] $, $u(z,t)$ is given  by
		\begin{equation} \label{charact1}
			u(z,t) = u (s, 0), 
		\end{equation}
		where the unique $s=s(z)$ is implicitly defined by the equation $F(z,s)=0$, with 
		\begin{equation} \label{charact1bis}
			F(z,s) \isdef 	z  -s+\frac{t}{u(s, 0)}.
		\end{equation}

		\item 
  \label{item3Mainthm} for each $t\in [0,1)$, there is a measure $\mu_t$, with $\supp(\mu_t) \subset\D_r$, such that
  \begin{equation}
      \label{u=Cauchy}
  u(z,t) =\mathcal C^{\mu_t}(z), \quad z \in \Omega_r.
  \end{equation}
  The moments
  \begin{equation}
      \label{u=moments}
  m_j(t) \isdef \int z^j \, d\mu_t(z), \quad j=0, 1, 2,\dots
  \end{equation}
  of $\mu_t$ are determined uniquely by $\mu_0$\footnote{\, Recall that $\mu_t$ itself not necessarily is uniquely defined by its moments.}, with $m_0(t)= 1-t$. Moreover, for $k\in \N$, $m_k(t)$ is a polynomial in $t$ of degree at most $k$, such that 
 \begin{equation}
     \label{limitMK}
\lim_{t\to 1} \frac{m_k(t)}{1-t}= m_1(0)^k, \qquad k\ge 1.
  \end{equation}
	\end{enumerate}
\end{thm}
In essence, this theorem guarantees that for any compactly supported initial positive measure $\mu_0$, the flow governed by the Hopf equation preserves analyticity and admits a representation as a Cauchy transform throughout its evolution.
\begin{remark} \label{remark22}
\begin{enumerate}[(a)]
\item Using the change of variables $w=u(s, 0)$ we can rewrite \eqref{charact1}--\eqref{charact1bis} as
\begin{equation} \label{charact2}
	u^{-1} (w, t) = u^{-1} (w, 0) -\frac{t}{w},
\end{equation}
where $u^{-1}$ is the functional inverse of $u$, or equivalently, as
 \begin{equation} \label{Shapiro}
 w=u\left(z+\frac{t}{w}, 0\right),
 \end{equation}
 which is an equation on $w=u(z, t)$.
 
\item	The assertion in  \textit{(\ref{item1Mainthm})} of Theorem~\ref{mainthm1} about $R$ depending only on $r$ and $T$ is in general false for a unit ($m_0=1$) but signed measure. A simple example is the case
	$$
	\mu = (a+1)\delta_0 - a\delta_1, \quad a>0,
	$$
	for which  
	\begin{equation}
	    \label{examplesigned}
     u(z,0)=\frac{a+1}{z}-\frac{a}{z-1}
	\end{equation}
	is holomorphic in $\Omega_1$, but
	$$
	u(z,t)= \frac{(1-2 t) z - (a+1-t)+\sqrt{z^2-2(t-1+a(1-2t))z + (a+1-t)^2}}{2 (z-1) z},
	$$
has branch points at $\zeta_+(t)=\overline{\zeta_-(t)}$, with $|\zeta_{\pm}|=a+1-t$. This shows that in this case there is no bounded $R$ that works for all values of the parameter $a>0$.

\item We can formally define the derivative with respect to $t$ of the measures $\mu_t$,
	$$
	\tau_t \isdef -\lim_{\varepsilon\to 0+} \frac{1}{\varepsilon}\left( \mu_{t+\varepsilon} -\mu_t \right).
 	$$
By \textit{(\ref{item3Mainthm})} of Theorem \ref{mainthm1}, $\tau_t (\C)=1$ for all $t\in (0,1)$. In the terminology of \cite{MR1618739}, Theorem \ref{mainthm1} states that $\tau_t$ is the \textit{inverse Markov transform} of the measure $\mu_t$. 

\item Formula \eqref{limitMK} can be interpreted as that $\mu_t$ disappears, as $t\to 1-$, as a mass point at $m_1(0)$, i.e. at the center of mass of the initial distribution $\mu_0$. This is the case, at least, when $\supp \mu_0\subset \R$.
\end{enumerate}
\end{remark}
 
Recall that the Hilbert transform of a measure $\mu$ on $\R$ is defined by
  \begin{equation} \label{defHillbertTransf}
		\mathcal H_\mu(x):= \frac{1}{\pi} \, \text{p.v.} \int \frac{d\mu(y)}{x-y}.
	\end{equation}
Theorem \ref{mainthm1} has the following consequence when the initial data is supported on the real line: the flow remains uniquely determined and real, and satisfies a nonlinear nonlocal transport equation.
\begin{cor}
 \label{itemNewMainthm} 
Let $ \mu_0$ be a compactly supported probability measure on $\R$, and let $\mu_t$, $t\in [0,1)$, be as described in (\ref{item3Mainthm}) of Theorem~\ref{mainthm1}. Then for each $t\in [0,1)$, the measure $\mu_t$ can be taken also supported on $\R$, and under this condition is unique.. If $ \mu_t $ is absolutely continuous with respect to the Lebesgue measure on $\R$ then the density
 		$$
 		f(x,t) \isdef \frac{d\mu_t(x)}{dx}
 		$$
 		satisfies
	\begin{equation} \label{nonlinearDiff}
 		 \partial_t\, f(x,t) = - \frac{1}{\pi}  \partial_x \arctan \left( \frac{\mathcal H_{\mu_t}(x)}{f(x,t)} \right) ,
 		\end{equation}
 	where $H_{\mu_t}$ is its Hilbert transform.
\end{cor}
The ``non-local transport equation'' \eqref{nonlinearDiff} coincides with the one found in \cite{MR4011508}.  

The initial value problem for the Hopf equation, whose existence and analyticity have been established in Theorem~\ref{mainthm1}, is directly related to the central problem of this paper, on the flow of zeros of a sequence of polynomials under iterated differentiation, as explained in the Introduction. 

Namely, recall that for a sequence of monic polynomials $Q_n$, with $\deg Q_n=n$, $n\in \N$, and the associated polynomials \eqref{def:QN}, we defined the normalized zero-counting measures \eqref{eq:defSigmaNK}, such that  
$$
\sigma_{n,k} \left( \C \right)= \frac{n-k}{n}. 
$$
The only assumption on $\{Q_n\}$ is that there exists $0<r<+\infty$ such that for each sufficiently large $n \in \N$, all zeros of $Q_n$ belong to $\D_r$, and by Gauss-Lucas' theorem, the same property is shared by all polynomials $Q_{n,k}$, defined in \eqref{def:QN}. Equivalently, we can say that for all sufficiently large $n \in \N$, measures $\sigma_{n,k}$, $k\in \{0, 1,\dots, n-1\}$, are supported in $\D_r$.

A key consequence of Theorem~\ref{mainthm1} is the existence of a universal value $R=R(T)>2r$ such that Hopf's equation \eqref{pde} has a holomorphic and uni-valued solution in $\Omega_R$, for any $0<t<T$, and for \textit{any} initial condition given by the Cauchy transform of \textit{any} probability measure supported in $\overline{\D_r}$. In particular, if for each $n\in \N$, we denote by $u_n(z,t)$ the solution of \eqref{pde} with the initial condition 
\begin{equation}
    \label{initialCondiU_n}
    u_n(z,0)\isdef \mathcal C^{\sigma_{n,0}}(z) =  \frac{  Q'_{n }(z)}{n\, Q_{n }(z)}. 
\end{equation}
then all $u_n(z,t)$ are guaranteed to exist and to be analytic in $\Omega_R \times [0,T]$.

As in \eqref{eqCauchyTransf00}, define for $k\in \{0, 1,\dots, n-1\}$, $n\in \N$,
\begin{equation} \label{eqCauchyTransf1}
v_{n,k} (z)\isdef	\mathcal C^{\sigma_{n,k}}(z)=	\frac{  Q'_{n,k}(z)}{n\, Q_{n,k}(z)} = \frac{n-k}{n} \, \frac{ Q_{n,k+1}(z)}{ Q_{n,k}(z)}  .
\end{equation}
For all such $k, n$, $v_{n,k} $ is holomorphic in $\overline{\Omega_r}$, and
\begin{equation}
    \label{estimateatinftyvnk}
v_{n,k} (z) = \frac{n-k}{n  }\, \frac{1}{  z } + \mathcal O(z^{-2}), \quad z\to \infty.
\end{equation}
Taking the logarithmic derivative in both sides of \eqref{eqCauchyTransf1} we obtain that
\begin{equation} \label{eqCauchyTransfNew}
	v_{n,k+1}(z)- v_{n,k}(z) =\frac{1}{n}\, \frac{\partial_z\, v_{n,k}(z)}{  v_{n,k}(z)}. 
\end{equation}
This identity can be regarded as a discrete version of Hopf's equation \eqref{pde}, and thus, it is natural to expect that under our assumptions, functions $v_{n,k}(z)$ and $u_n(z,k/n)$ are close for sufficiently large $n$'s.

If additionally
\begin{equation} \label{assumptionWeakConvergence}
	\sigma_{n,0}  \stackrel{*}{\longrightarrow} \mu_0, \quad n\to \infty,
\end{equation}
where $\stackrel{*}{\longrightarrow}$ denotes the weak-* convergence and $\mu_0$ is a  probability measure on $\C$ supported on $\overline{\D_r}$, then the initial conditions \eqref{initialCondiU_n} converge to the initial condition \eqref{pdeInitial}, at least in the fixed neighborhood of infinity, and we can expect that the solutions $u_n(z,t)$ of the equation \eqref{pde} with the initial conditions \eqref{initialCondiU_n} will converge there to $u(z,t)$, the solution of the same equation \eqref{pde}, but with the initial condition \eqref{pdeInitial}. These heuristic arguments can be made rigorous, although the trivially looking assertion about the proximity of the solutions to the ``discrete'' and ``continuous'' Hopf equations requires some careful analysis due to the nonlinear character of these equations. For convenience of the reader, we formulate Theorem~\ref{mainthm0} here, in a more precise form:
\begin{thm} \label{mainthm2}
 Let $Q_n$ be a sequence of polynomials with uniformly bounded zeros such that \eqref{assumptionWeakConvergence} holds. 
Then for any $t \in[0,1)$, there exists a positive Borel measure $\mu_t$ with $\supp \mu_t \subset \conv (\supp \mu_0)$ and $\mu_t(\C)=1-t$, such that in a neighborhood of infinity,
\begin{equation} \label{coincidenceCauchy}
 \mathcal C^{\mu_t}(z) =u(z,t),
\end{equation}
where $u$  is the solution of equation \eqref{pde} with the initial condition \eqref{pdeInitial}. 

Furthermore, for any sequence $k_n$ of natural numbers with $k_n / n \rightarrow t$ we have that
	$$
	\lim_n  \mathcal C^{  \sigma_{n,k_n}}(z) =\mathcal C^{\mu_t}(z)
	$$
	locally uniformly in $\C\setminus \left( \conv (\supp \mu_0)\right)$.
 
	Finally, if all zeros of $Q_n$ are real, then $\mu_t$ is uniquely determined from $u$, and  
	$$
	\sigma_{n, k_n} \stackrel{*}{\longrightarrow} \mu_t, \quad n \to \infty. 
	$$
\end{thm}
\begin{remark}
In the statement above, $ \conv (K)$ denotes the convex hull of the set $K$. 

With the assumptions of the theorem, let $T=(1+t)/2$, and let $R=R(T)$ be as in Theorem~\ref{mainthm1}. Then to establish Theorem~\ref{mainthm2}, it is sufficient to prove that
$$
\lim_n  \mathcal C^{  \sigma_{n,k_n}}(z) =u(z,t)
$$
locally uniformly in $\Omega_R =\{z\in \C:\, R<|z|<+\infty\}=\C \setminus \overline{\D_R} $.  Indeed, by the weak-* compactness of the sequence $\sigma_{n, k_n} $, there exists a measure $\mu_t$, with the properties indicated above, such that $\sigma_{n, k_n} \stackrel{*}{\longrightarrow} \mu_t$ along a subsequence of $\N$, and thus, \eqref{coincidenceCauchy} holds in $\Omega_R$. The rest of the assertion follows by elementary analytic continuation arguments.

As a consequence, the  Cauchy transform of the measure $\mu_t$ is completely determined in the complement of $ \conv (\supp \mu_0)$. However, in its maximum generality, this does not determine measure $\mu_t$ uniquely, which makes the statement of Theorem~\ref{mainthm2} apparently less straightforward. 
\end{remark}

Finally, we recast Theorem~\ref{mainthm2} in the terminology of free probability; see e.g.~\cite{Arizmendi21, MR4408504, MR3585560, MR2266879} for a background. Along with the Cauchy transform $\mathcal C^\mu$ of a compactly supported probability measure $\mu$, as in \eqref{def:Cauchy}, we can define its \textit{$R$-transform} implicitly by  
	\begin{equation}
	    \label{defRtransform}
 \mathcal R_\mu\left(\mathcal C^\mu(z)\right) = z-\frac{1}{\mathcal C^\mu(z)} ;
	\end{equation}
in particular, $\mathcal R_\mu$ is an analytic function in a neighborhood of the origin. For $s>0$, 
the \textit{fractional free additive convolution}  $\mu^{\boxplus s}$ of $\mu$ is defined via the identity (see \cite{ShlyakhtenkoTao21})
	\begin{equation}
	    \label{propRtransform}
	\mathcal R_{\mu^{\boxplus s}}(w) = s \mathcal R_\mu (w). 
\end{equation}
\begin{cor}  \label{thm:Steiner}
Let $Q_n$ be a sequence of polynomials as in Theorem~\ref{mainthm2}, such that \eqref{assumptionWeakConvergence} holds, where $\mu_0$ is a compactly supported probability measure on $\R$. For $t\in [0,1)$, denote
 $$
 \widehat Q_{n,k}(x) \isdef Q_{n,k}((1-t) x), 
 $$
 where we use the notation \eqref{def:QN}. 
 
Then for $k/n\to t$,
 \begin{equation} \label{limitFreeConv}
\frac{1}{n-k} \, \chi\left(\widehat  Q_{n,k} \right)  \stackrel{*}{\longrightarrow} \mu_0^{\boxplus 1/(1-t)}.
 \end{equation}
 \end{cor}
As mentioned in the Introduction, \eqref{limitFreeConv}, established heuristically in \cite{Steiner2021}, was rigorously proven in \cite{HoskinsKabluchko21}. Later, a short proof was presented in \cite{Arizmendi21}, based entirely on arguments from free probability. 
We will show that this result is an immediate consequence of solving Hopf's equation by the method of characteristics.

\section{Proof of Theorem~\ref{mainthm1} and Corollary~\ref{itemNewMainthm}} \label{sec:proofThm1}

We start by addressing the evolution in a neighborhood of infinity of the Cauchy transform \eqref{def:Cauchy} of a finite Borel \textit{signed} measure $\mu$, compactly supported on $\C$, according to the Hopf equation \eqref{pde}. 
Recall that $\mathcal C^\mu$ is a holomorphic function in $\C \setminus \supp(\mu)$, and that the series in \eqref{seriesCauchyTransform} converges locally uniformly in any $\Omega_r=\{z\in \C:\, r<|z|<+\infty\}$, disjoint from $\supp(\mu)$. We assume additionally that 
\begin{equation}
    \label{normCondition}
m_0=\mu(\C)=1,
\end{equation}
and denote by $|\mu|$ the total variation of $\mu$; since $\mu$ is signed, in general, $|\mu| (\C)\geq m_0$.

As a first step, we need some uniform estimates on $\mathcal C^\mu$:
\begin{prop}
    \label{prob:boundsCauchy}
    Let $\mu$ be a finite Borel and compactly-supported signed measure on $\C$, such that \eqref{normCondition} holds, and let  $0<r<+\infty$ be such that $\supp (\mu) \subset\D_r$. Then for any $R>  (1+|\mu|(\C))\, r$,  $\mathcal C^\mu$ is analytic in $\overline{\Omega_R}$ and does not vanish for $z\in \Omega_R$.

If in addition $\mu$ is a positive (probability) measure, then for $|z|\ge R>2r$,
    \begin{align}
        \label{boundsCauchy1}
    & \frac{2}{9}< \frac{1-r/R}{(1+r/R)^2} \le \left|  z\mathcal C^\mu(z)\right| \le \frac{R}{R-r} ,
    \\  
    & \frac{1-2r/R}{(1+r/R)^4} \le \left| z^2\, \frac{d}{dz}\mathcal C^\mu(z)\right| \le \frac{1}{(1-r/R)^2}<4.
\label{boundsCauchy2}
    \end{align}
Moreover, if $R>(\sqrt{2}+1)r$, then $C^\mu$ is univalent (injective) in  $\Omega_R$, and 
\begin{equation} \label{image}
\overline{\D_{1/(4R)}}\subset C^\mu\left( \overline{\Omega_R} \right)\subset \overline{\D_{1/(R-r)}}.
\end{equation}
\end{prop}
\begin{remark}
    As formula \eqref{examplesigned} in Remark \ref{remark22} (b) shows, for a signed measure $\mu$, $R$ in the first assertion of Proposition~\ref{prob:boundsCauchy}, in general, cannot be made independent of $|\mu|$.
\end{remark}
\begin{proof}
For an arbitrary $R>r$, function $z\, \mathcal C^\mu(z)$ is analytic in $\overline{\Omega_R}$ and can be written as
$$
z\mathcal C^\mu(z) = 1 +    \int \frac{y}{z-y}\, d\mu(t).
$$
Since for $|z|=R$,
$$
\left| \int \frac{y}{z-y}\, d\mu(t)\right|\le \frac{r}{R-r}|\mu(\C)|,
$$
if $R>r$ is such that 
 $$
	\frac{r}{R- r} |\mu|(\C)<1 \quad \text{or, equivalently,} \quad    R>r (1+|\mu|(\C)),
	$$
then by Rouch\'e's theorem, $\mathcal C^\mu(z)\ne 0$ in $\overline{\Omega_{R}}$. 
In particular, if $\mu$ is a probability measure, this is true for any $R>2r$. 

Assuming that $\mu$ is a probability measure with $\supp (\mu) \subset\D_r$, we have $ |\mu|(\C)=1$, and the bound obtained above gives us that for $|z|=R>2r$,
$$
\left| z\mathcal C^\mu(z) \right| =\left| 1 +    \int \frac{y}{z-y}\, d\mu(t)\right| \le 1 + \frac{r}{R-r}=\frac{R}{R-r}.
$$
On the other hand,
$$
\left| z \, \mathcal C^\mu(z)\right| \ge \Re \left( z\,  \mathcal C^\mu(z)\right) = \int \Re \left(\frac{1}{1-y/z} \right)\, d\mu(y),
$$
which gives us that for $|z|=R>2r$ and $|y|\le r<R$,
$$
\Re \left(\frac{1}{1-y/z} \right) =\frac{1-\Re (y/z)}{|1-y/z|^2}  \ge \frac{1-r/R}{(1+r/R)^2}>\frac{2}{9},
$$
and the bounds in \eqref{boundsCauchy1} follow. 

In a similar fashion,
$$
\left| z^2 \,\frac{d}{dz} \mathcal C^\mu(z)\right| = \left|  \int \frac{1}{(1-y/z)^2}\, d\mu(t) \right| \le  \frac{1}{(1-r/R)^2} <4,
$$
and
$$
\left| z^2 \,\frac{d}{dz}  \mathcal C^\mu(z)\right| \ge \Re \left( -z^2 \,\frac{d}{dz}  \mathcal C^\mu(z)\right) = \int \Re \frac{1}{\left(1-y/z\right)^2} \, d\mu(y).
$$
For $|z|\ge R>2r$ and $y\in\D_r$, 
\begin{align*}
\Re   \left(1-y/z\right)^2  & = 1- 2\Re \left(\frac{y}{z}\right) + \Re \left(\frac{y}{z}\right)^2 = 1- 2\Re \left(\frac{y}{z}\right) + \left(\Re 
 \frac{y}{z}\right)^2 - \left(\Im 
 \frac{y}{z}\right)^2 \\
 & = \left(1- \Re 
 \frac{y}{z}\right)^2 - \left(\Im 
 \frac{y}{z}\right)^2\ge \left(1- \Re 
 \frac{y}{z}\right)^2 -\left|
 \frac{y}{z}\right|^2\ge \left(1- \frac{r}{R}\right)^2 -\left(
 \frac{r}{R}\right)^2 \\
 & = 1-   \frac{2r}{R} >0,
\end{align*}
so that
$$
\Re   \frac{1}{\left(1-y/z\right)^2} = \frac{\Re   \left(1-y/z\right)^2}{\left|1-y/z\right|^4} \ge \frac{1-2r/R}{(1+r/R)^4},
$$
and hence,
$$
\left| z^2 \,\frac{d}{dz}  \mathcal C^\mu(z)\right| \ge  \frac{1-2r/R}{(1+r/R)^4},\qquad |z|\ge R>2r.
$$
This concludes the proof of \eqref{boundsCauchy2}.

Furthermore, for $R>r$, consider the function
\begin{equation}
    \label{defF1}
    f_R (z) \isdef R\, \mathcal C^\mu\left(\frac{R}{z}\right)= z \int \frac{1}{1-z y/R} \, d\mu(y)= z+ \mathcal O\left( z^2 \right), \quad z\to 0,
\end{equation}
defined in the unit disk $D_1=\{z\in \C:\, |z|<1\}$. 
We have that
$$
\Re \left( \frac{f_R(z)}{z} \right) = \int \Re \left( \frac{1}{1-z y/R}\right) \, d\mu(y) = \int   \frac{  1-\Re \left( z y\right)/R}{|1-z y/R|^2}  \, d\mu(y) .
$$
Since
$$
 1-\frac{\Re \left( z y\right)}{R}\ge 1-\frac{  \left| z y\right|}{R}\ge 1-\frac{  r}{R}>0,
$$
we conclude that $\Re(f_R(z)/z)>0$ in $|z|<1$, and by \cite[Theorem 3]{MR0140674}, $f_R$ is univalent in $|z|<\sqrt{2}-1$ (or, equivalently, $\mathcal C^\mu$ is univalent in $\Omega_{(\sqrt{2}+1)R}$).

Moreover, with $R>(\sqrt{2}+1)r$, by the Koebe $1/4$ theorem (see, e.g.~\cite[\S 2.2]{MR0708494}), the image $f_R(\D_1)$ of the unit disk by the function $f_R$ contains the disk $\overline{\D_{1/4}}$, or in other words, the image of $\overline{\Omega_R}$ by $\mathcal C^\mu$ contains the disk $\overline{\D_{1/(4R)}}$.
\end{proof}

It is well known that the Hopf equation \eqref{pde} can be solved by the method of characteristics. In preparation to the analysis of this solution, we prove the following auxiliary result: 
\begin{prop} \label{prop:solutionPDE}
Let $\mu$ be a finite Borel signed measure $\mu$ on $\C$, such that \eqref{normCondition} holds. 
Assume $0<r<+\infty$ is such that $\supp \mu \subset\D_r$, and let $0<T<1$ be fixed. Then there exists an $R> 2r$, depending only on the values of $T$, $|m_1|$, and $|\mu| (\C)$, such that 
\begin{equation*}
    f_1 (s) = \mathcal C^\mu\left(\frac{1}{s}\right)= \int \frac{s}{1-s y} \, d\mu(y)
\end{equation*}
is analytic in $\overline{\D_{1/R}}$, does not vanish for $0<|s|\le  1/R$, and for each $z\in \D_{1/(2R)}$ and each $t\in [0,T]$, equation 
	\begin{equation} \label{main}
		z -  s- \frac{z s t}{f_1 (s)}=0
	\end{equation}
	has a unique solution $s\in \D_{1/R}$. 
\end{prop}
\begin{proof}
By Proposition \ref{prob:boundsCauchy}, if $R>r (1+|\mu|(\C))$ then $f_1(z)$ is analytic and non-vanishing for $0<|z|<1/R$. The assertion about the unisolvence of the equation \eqref{main} is obviously true for $z=0$, so we assume that $z\neq 0$.

Write
\begin{equation}
    \label{f-g}
    f_1 (s)=s + s^2 g (s),
\end{equation}
with
\begin{equation}
    \label{defFG}
    g (s) \isdef \int \frac{y}{1-s y}\, d\mu(y),
\end{equation}
so that 
$$
g (0) = \int y \, d\mu(y) = m_1,
$$
the first moment of $\mu$, see \eqref{moments}. 
Notice that  
	$$
	|g (s)-g (0)| = \left|  \int \frac{y}{1-s y}\, d\mu(y) - \int y \, d\mu(y)\right|  =|s| \left|  \int \frac{y^2}{1-s y}\, d\mu(y) \right| ,
	$$
	and for $|s|=1/R$,
	$$
	|g (s)-g (0)| = \frac{1}{R}\,  \left|  \int \frac{y^2}{1-s y}\, d\mu(y) \right| \le \frac{r^2}{R- r} |\mu|(\C).
	$$
	If we take $R>r (1+|\mu|(\C))$ such that
	$$
	\frac{r^2}{R- r} |\mu|(\C)<1 \quad \text{or, equivalently,} \quad    R>r (1+r|\mu|(\C)),
	$$
	then
	\begin{equation} \label{a1}
		|g (s)|\le 1+|m_1|, \quad s\in \overline{\D_{1/R}}.
	\end{equation}
	
	Now, given $T\in (0,1)$, let $ R>\max\{ r (1+|\mu|(\C)), r (1+r|\mu|(\C))\}$ be such that 
	$$
	R>\left(1+|m_1| \right)\frac{1+T}{1-T} \quad \text{or, equivalently,} \quad  0<\frac{T}{1- ( |m_1|+1)/R}<\frac{1+T}{2}<1.
	$$
	Then for $s\in D_{1/R}$,
	$$
	0\le \frac{|g(s)|}{R } \le \frac{1+|m_1|}{R }<\frac{1-T}{1+T}<1,
	$$
	so that
	\begin{equation} \label{eq:bound1}
			\left| \frac{t}{1+ sg (s)}\right| \le \frac{T}{|1- |g (s)|/R |} \le \frac{T}{1- |g (s)|/R } < \frac{1+T}{2}<1,
	\end{equation}
 and in consequence, for $z\neq 0$ and $|s|=1/R$,
 	\begin{equation}
    \label{boundsErrorProp1}
   \left|   \frac{z st}{f_1 (s)}\right| =   \left|   \frac{z t}{1+ sg (s)}\right| < |z|.
	\end{equation}
  Thus, since for $0<|z|<1/(2R)$ and $|s|=1/R$ we have 
  $$
  |z-s|>\frac{1}{2R}>|z|,
  $$
  we can apply Rouch\'e's theorem to claim that for every $z\in \D_{1/(2R)}\setminus \{0\}$, the equation \eqref{main} has a unique solution. 
  
All assertions have been proved for any
	$$
	R >   \max \left\{    r (1+|\mu| (\C)),  r (1+r|\mu|(\C)), \left(1+|m_1| \right)\frac{1+T}{1-T}  \right\}\ge 2r,
	$$
which concludes the proof. 
\end{proof}

\begin{remark}
Under the assumption of the Proposition,  for a signed measure $\mu$, both  $|\mu| $  and $|m_1|$ can be arbitrary large even though $m_0=\mu(\C)=1$.

However, if $\mu$ is a probability measure, then $|\mu|(\C)=m_0=1$ and 
	$$
	\supp(\mu)\subset\D_r \quad \Rightarrow \quad |m_1| =\left| \int t \, d\mu(t) \right|  \le r,
	$$
showing that we can take
	\begin{equation}
	    \label{boundForR}
     R =R(T)>    \max \left\{  2r,   r(1+r)  ,  \left(1+r \right)\frac{1+T}{1-T}  \right\}, 
	\end{equation}
 which depends only on the values of $T$ and $r$. 
\end{remark}

Now we are ready to prove the first two statements of Theorem \ref{mainthm1}.

\begin{proof}[Proof of statements \textit{\eqref{item1Mainthm}} --\textit{\eqref{item2Mainthm}} of Theorem \ref{mainthm1}]

Let $0<T<1$ and $0<r<+\infty$ be fixed and take $R=R(T)>2r$, satisfying \eqref{boundForR}. Consider equation \eqref{pde} with the initial condition \eqref{pdeInitial}, where $\mu_0$ is a  probability measure supported by $\D_r$. We solve \eqref{pde} using the standard method of characteristics, see e.g.~\cite{salih2015inviscid, MR1699025}.
	
	If $z=z(t)$ is differentiable, then
	$$
	\frac{d}{dt}\,   u(z(t),t) =\frac{\partial u}{\partial t} + z'(t)  \frac{\partial u}{\partial z}.
	$$
	Comparing it with \eqref{pde} we see that if $z(t)$ satisfies
	\begin{equation} \label{identZ}
		z'(t)=-\frac{1}{u(z(t),t)},
	\end{equation}
	then for a solution $u$ of \eqref{pde}  we should have
	\begin{equation} \label{ident0}
		\frac{d}{dt} u\left( z(t), t\right)=0.
	\end{equation}
	Differentiating \eqref{identZ} with respect to $t$ again and using \eqref{ident0} we conclude that $z''(t)\equiv 0$, so that $z(t)$ must be linear in $t$, $z(t)=Lt+s$. 
	
	To find the slope $L$, we use the initial values,
	$$
	L=z^{\prime}(0)=-\frac{1}{u(z(0), 0)}=-\frac{1}{u(s,0)},
	$$
	so that $z(t)=s-t/u(s, 0)$. This is the characteristics line, along which $u(z, t)$ is constant, with the constant given by
	$u(z(0), 0)=u(s, 0)$. Thus, for $(z, t)$, the solution of \eqref{pde} is $	u(z,t) = u (s, 0)$, where $s=s(z)$ is  the solution of  
	$$ 
	z  -s+\frac{t}{u(s, 0)}=0,
	$$
	that is, of $F(z,s)=0$. Making the change of variables $z \mapsto 1/z$ and $s \mapsto 1/s$ this is equivalent to write 
	\begin{equation} \label{charact1a}
		u(1/z,t) = u (1/s, 0), 
	\end{equation}
	where $s=s(z)$ is defined implicitly by 
	\begin{equation} \label{charact1b}
		z  -s-\frac{zs t}{u(1/s, 0)}=0.
	\end{equation}
	Thus, it remains to apply Proposition~\ref{prop:solutionPDE} with 
	$	f_1(z)=u(1/z,0) $ to conclude the proof.
\end{proof}

Now we turn to the proof of the statement \textit{(\ref{item3Mainthm})} of Theorem \ref{mainthm1}. The fact that there exists a measure $\mu_t$ such that \eqref{u=Cauchy} holds will follow directly from Theorem \ref{mainthm2}.
Hence, we only need to establish the second part of the statement, on the properties of the moments $m_k(t)$. Notice that the fact that they are uniquely determined by the initial data $\{m_k(0)\}_{k\ge 0}$ is also a consequence of the uniqueness of the solution $u(z,t)$, already proved before.

\begin{proof}[Proof of statement \textit{(\ref{item3Mainthm})} of Theorem \ref{mainthm1}]
Assume that $\mu_0$ is a probability measure compactly supported in the disk $\D_r$, $0<T<1$, and take $R=R(T)>2r$, satisfying \eqref{boundForR}. The unique solution $u(z,t)$ of the equation \eqref{pde} for $t\in [0,T]$, satisfying the initial condition \eqref{pdeInitial}, has the Laurent expansion, convergent in $\Omega_R$, given by
	\begin{equation} \label{expansionatInfinity}
		u(z,t) = \sum_{k=0}^\infty \frac{m_k(t)}{z^{k+1}}, \quad |z|>R,
	\end{equation}
	with $m_0(0)=1$. 
	Replacing this expansion into \eqref{pde}, rewritten as
	\begin{equation} \label{pdeRewritten}
		\partial_z\, u =u\, \partial_t\,  u,
	\end{equation}
	we get  an infinite system of equations
	\begin{equation} \label{mn}
		\sum_{j=0}^{k} m_{k-j}(t) m'_j(t) =-(k+1) m_k(t), \quad k=0, 1, \dots,
	\end{equation}
	where we have denoted
	$$
	m'_k(t) = \frac{d}{d t}m_k(t) .
	$$
	For $k=0$, we immediately get that $m_0'(t)=-1$, so that $m_0(t)= 1-t$. For $k=1$, we obtain $m_1(t)=m_1(0)(1-t)$, so that \eqref{limitMK} is valid for $k=1$.
	
	We conclude the proof by induction, assuming that the assertion is established for $m_1(t), \dots, m_{k-1}(t)$, with $k\ge 2$. We need to show that $m_k(t)$ is a polynomial in $t$ of degree at most $k$, whose coefficients depend only on the initial moments $m_j(0)$, $0\leq j \le k$, and such that
 $$  
 m_k(1)=0, \quad -m'_k(1)=\lim_{t\to 1} \frac{m_k(t)}{1-t}= m_1(0)^k.
$$
 
Let us rewrite \eqref{mn} as
	\begin{equation} \label{differMn}
		(1-t) m_k'(t) +k \, m_k(t) =-  \sum_{j=1}^{k-1} m_{k-j}(t)  m_j'(t)=-f_k'(t),
	\end{equation}
where
	$$
	f_k(t)\isdef   \frac{1}{2}\, \sum_{j=1}^{k-1} m_{k-j}(t)  m_j(t).
	$$
 Multiplying both sides of \eqref{differMn} by the integrating factor $(1-t)^{-k-1}$ we get
	$$
	\left(  \frac{m_k(t)}{(1-t)^{k}} \right)' =-  \frac{f_k'(t)}{ (1-t)^{k+1}} , 
	$$
	so that with the account of initial conditions, 
	\begin{equation} \label{eq:mkinterm}
		m_k(t) = (1-t)^{k}m_k(0) - (1-t)^{k} \int_0^t \frac{f_k'(u)}{ (1-u)^{k+1}} \, du.
	\end{equation}

Notice that
 $$
 f_2(t)= \frac{1}{2} m_1(t)^2= \frac{1}{2} m_1(0)^2(1-t)^2, \quad f''_2(1)=m_1(0)^2,
 $$
and by the induction hypothesis, $f_k$ is a polynomial of degree at most $k$ in $t$, with 
\begin{equation}
    \label{inductionFk}
    \begin{split}
f_k(1)& =f_k'(1)=0, \\ 
f_k''(1) &=\lim_{t\to 1}\frac{2f_k(t)}{(1-t)^2} =-\lim_{t\to 1}\frac{f'_k(t)}{1-t}  = \sum_{j=1}^{k-1} m'_{k-j}(1)  m'_j(1)=(k-1)m_1(0)^k.
\end{split}
\end{equation}
Integrating in \eqref{eq:mkinterm} by parts $k$ times and using that $f^{(k+1)}\equiv 0$, we get that
\begin{align*}
		\int_0^t   \frac{f_k'(u)}{ (1-u)^{k+1}} \,  du & =\left( \sum_{j=1}^{k} \frac{(-1)^{j-1}(k-j)!}{k!}\frac{f_k^{(j)}(u)}{(1-u)^{k+1-j}  }\right)\bigg|_0^t,
	\end{align*}
	so that
	\begin{align*}
		- (1-t)^{k} \left[  \int_0^t   \frac{f_k'(u)}{ (1-u)^{k+1}} \,  du\right] & =  (1-t)^{k} d  
		+    \sum_{j=1}^{k} \frac{(-1)^{j} (k-j)!}{k!}\, f_k^{(j)}(t)\, (1-t)^{j-1},
	\end{align*}
	where $d$ is a constant:
	$$
	d=   \sum_{j=1}^{k} \frac{(-1)^{j-1} (k-j)!}{k!} \, f_k^{(j)}(0).
	$$
	Combining all in \eqref{eq:mkinterm}, we conclude that
  \begin{align*}
       m_k(t) &  = m_k(0) (1-t)^{k} +  \sum_{j=1}^{k} (-1)^{j} \frac{(k-j)!}{k!}\, \left[ f_k^{(j)}(t) - f_k^{(j)}(0)\, (1-t)^{k-j+1} \right]\, (1-t)^{j-1}\\
        &= - \frac{  1}{k}\, f'_k (t)  + \frac{ 1}{k(k-1)}\, f''_k (t)\, (1-t) + \mathcal O((1-t)^2), \quad t\to 1.
  \end{align*}
By \eqref{inductionFk}, $f'_k(1)=0$, so that $m_k(1)=0$, and
\begin{align*}
-m_k'(1) & =\lim_{t\to 1}\frac{m_k(t)}{1-t}=  \lim_{t\to 1} \left( - \frac{  1}{k}\, \frac{f'_k (t)}{1-t}  + \frac{ 1}{k(k-1)}\, f''_k (t) \right)\\
& = \frac{f_k''(1)}{k-1} =  m_1(0)^k ,
\end{align*}
which proves the assertion.
\end{proof}

\begin{proof}[Proof of Corollary~\ref{itemNewMainthm}]
  Assume that $ \mu_t $ is absolutely continuous with respect to the Lebesgue measure on $\R$ with the density
 		$$
 		f(x,t) \isdef \frac{d\mu_t(x)}{dx}.
 		$$
The Sokhotski-Plemelj formulas 
  \begin{equation} \label{sokhotsky}
		\left(\mathcal C^\mu\right)_{\pm} (x) =  \lim_{y\to 0+} \mathcal C^\mu(x\pm i y) = -\pi \left( \mathcal H_{\mu}(x) \pm i \mu'(x) \right), \quad x \in \R,
	\end{equation}
where $\mathcal H_\mu$ is the Hilbert transform 
defined in \eqref{defHillbertTransf}, imply that,  
$$
u_{\pm} (x,t) = -\pi  (1-t) \left( \mathcal H_{\mu_t}(x) \pm i f \right).
$$
Using it in \eqref{pde} and subtracting both boundary values we get the identity
$$
2\pi i \, \partial_t\, \left[(1-t) f(x,t)\right]= 2i \, \partial_x  \left( \frac{f(x,t)}{\mathcal H_{\mu_t}(x)} \right) \frac{1}{1+ \left( \frac{f(x,t)}{\mathcal H_{\mu_t}(x)} \right)^2 },
$$
or 
$$
\partial_t\, \left[(1-t) f(x,t)\right]= \frac{1}{\pi}  \partial_x\, \arctan \left( \frac{f(x,t)}{\mathcal H_{\mu_t}(x)} \right) = - \frac{1}{\pi} \partial_x\, \arctan \left( \frac{\mathcal H_{\mu_t}(x)}{f(x,t)} \right) ,
$$
which proves \eqref{nonlinearDiff}. 
\end{proof}

\section{Proof of Theorem~\ref{mainthm2} and Corollary~\ref{thm:Steiner}} \label{sec:proofThm2}

The triangular table of monic polynomials $Q_{n,k}$, $\deg Q_{n,k}=n-k$, $n\in \N$,  
$$  
Q_{n,k}(z) =  \frac{(n-k)!}{n!}\, \frac{d^k}{dz^k } Q_n(z) , \quad k\in \{0, 1,\dots, n-1\},
$$  
was introduced in \eqref{def:QN}, along with their zero-counting measures $\sigma_{n,k}$ in \eqref{eq:defSigmaNK}. 

The heuristic arguments supporting Theorem~\ref{mainthm2} were presented in Section~\ref{sec:main}, where we also introduced the functions $u_n(z,t)$, the solution of Hopf's equation \eqref{pde} with the initial condition \eqref{initialCondiU_n}. As explained, our goal is to show that 
$$
u_{n,k}(z)\isdef u(z,k/n)
$$
and $v_{n,k}$, defined in \eqref{eqCauchyTransf1},  
are asymptotically close. The solution of the Hopf equation is characterized by the identity \eqref{charact2}, which motivates us to study the behavior of their functional inverses. 

Since $v_{n,k} $ is the Cauchy transform of a measure of size $1-k/n$, by Proposition~\ref{prob:boundsCauchy}, the inverse function $v_{n,k}^{-1}(\zeta)$ is guaranteed to exist for all $(\zeta,k/n)\in \mathbb D_\rho \times [0,T]$, $0<T<1$, where $\rho = (1-T)/(4R)$, and $R>(\sqrt{2}+1)r$; moreover, $v_{n,k}^{-1}(\mathbb D_\rho)\subset \Omega_R$.

On the other hand, recall that 
$$	
u_n(z,t) = u_n (s, 0) = \mathcal C^{\sigma_{n,0}}(s) =  \frac{  Q'_{n }(s)}{n\, Q_{n }(s)},
$$
and by Proposition~\ref{prop:solutionPDE}, the inverse function $u_{n,k}^{-1}(\zeta)$ is guaranteed to exist for all $(\zeta,k/n)\in \mathbb D_{1/(4R)} \times [0,T]$, $0<T<1$, where $R>(\sqrt{2}+1)r$; also, $u_{n,k}^{-1}(\mathbb D_\rho)\subset \Omega_R$.

In summary, given $R>(\sqrt{2}+1)r$ and $0<T<1$, there exists a $0<\rho<(1-T)/(4R)$ such that we can consider the inverse functions $u_{n,k}^{-1}(\zeta )$ and $v_{n,k}^{-1}(\zeta )$, conformal mappings of $\mathbb D_\rho$ into $\Omega_R$ for every $k/n\in [0,T]$. We fix these values of $R$ and $\rho$ for the rest of the proof. 

\begin{prop}
    \label{lem5}
For the inverse functions $v_{n, k}^{-1}(\zeta)$ and $v_{n, k+1}^{-1}(\zeta)$, with $\zeta \in \D_\rho$, we have
\begin{equation}
    \label{eq:12}
v_{n,k+1}^{-1}\left(\zeta\right)-v_{n, k}^{-1}(\zeta)=-\frac{1}{n \zeta}  +\Delta_{n, k}(\zeta).
\end{equation}
There is $\rho'<\rho$ and $M>0$, depending only on $r$, such that $\Delta_{n, k}$ are analytic in $\D_{\rho'}$ and 
\begin{equation}
    \label{boundDelta}
    \left|\Delta_{n, k}(\zeta)\right| \leq \frac{M}{n^2 }, \quad \zeta \in \D_{\rho'}.
\end{equation}
\end{prop}
\begin{proof}
Notice that by \eqref{estimateatinftyvnk}, 
$$
v_{n,k+1}^{-1}\left(\zeta\right)-v_{n, k}^{-1}(\zeta) +\frac{1}{n \zeta} 
$$
is analytic at $\zeta=0$, so it is enough to establish the bound \eqref{boundDelta} on a circle $|\zeta|=\rho'$.

To simplify the notation, denote $\varphi_{n, k}(\zeta)\isdef v_{n, k}^{-1}(\zeta)$ and let $^\prime$ stand for the derivative with respect to the corresponding variable that should be clear in each context: $v_{n, k}'(z)=\partial_z \,   v_{n, k}(z)$, $\varphi_{n, k}'(\zeta)=\partial_\zeta \,   v_{n, k}^{-1}(\zeta)$, etc. We also denote $z_k\isdef \varphi_{n, k}(\zeta)$ and $z_{k+1}\isdef \varphi_{n, k+1}(\zeta)$. 

Substituting $z_k=\varphi_{n, k}(\zeta)$ in \eqref{eqCauchyTransfNew} and using that $ v_{n, k}'(z_k)=1/ \varphi_{n, k}'(\zeta)$, we obtain
$$
v_{n, k+1}\left( \varphi_{n, k}(\zeta) \right)
=\zeta +\frac{1}{n \zeta} \frac{1}{\varphi_{n, k}'(\zeta) }.
$$
Applying now $\varphi_{n, k+1}=v_{n, k+1}^{-1}$ to both sides yields
$$
\varphi_{n, k}(\zeta) =\varphi_{n, k+1}\left( \zeta +\frac{1}{n \zeta} \frac{1}{\varphi_{n, k}'(\zeta) }\right),
$$
and thus,
\begin{equation}
    \label{firstidentityPhi1}
\varphi_{n, k+1}(\zeta)- \varphi_{n, k}(\zeta) =  \varphi_{n, k+1}(\zeta) - \varphi_{n, k+1}\left( \zeta +\frac{1}{n \zeta} \frac{1}{\varphi_{n, k}'(\zeta) }\right).
\end{equation}
Using that for small $s\in \C$,
\begin{equation}
    \label{taylorEx}
 \varphi_{n, k+1}\left( \zeta +s\right)= \varphi_{n, k+1}\left( \zeta \right)+ \varphi_{n, k+1}'\left( \zeta \right) s + \frac{1}{2}\varphi_{n, k+1}''\left( \zeta +\theta \right) s^2,
\end{equation}
we get that 
\begin{equation}
    \label{identityPhi1}
    \begin{split}    
  z_{k+1}-z_k & =   \varphi_{n, k+1}(\zeta)- \varphi_{n, k}(\zeta)  =  -\frac{1}{n \zeta} \frac{\varphi_{n, k+1}'\left( \zeta \right)}{\varphi_{n, k}'(\zeta) }-\frac{\delta_{n, k}(\zeta)}{n^2 \zeta^2} \\
&  = -\frac{1}{n \zeta} -\frac{1}{n \zeta} \frac{\varphi_{n, k+1}'\left( \zeta \right)-\varphi_{n, k}'\left( \zeta \right)}{\varphi_{n, k}'(\zeta) }-\frac{\delta_{n, k}(\zeta)}{n^2 \zeta^2}. 
\end{split}
\end{equation}

With the notation introduced above, let us rewrite
\begin{equation} \label{identitiesderivativesPhi}
\frac{\varphi_{n, k+1}'\left( \zeta \right)-\varphi_{n, k}'(\zeta)}{\varphi_{n, k}'(\zeta)  }  = \frac{ v_{n, k}' \left( z_k \right) - v_{n, k+1}' \left( z_{k+1} \right) }{  v_{n,k+1}'\left( z_{k+1} \right)}  =  \frac{1}{n}\, \varepsilon_{n,k}^{(1)} (\zeta) - \varepsilon_{n,k}^{(2)} (\zeta)(u-z) ,
\end{equation}
where, again by \eqref{eqCauchyTransfNew}, 
\begin{equation}
\begin{split}
\varepsilon_{n,k}^{(1)} (\zeta) & = n \, \frac{  v_{n, k}'\left( z_{k+1}\right)-  v_{n, k+1}'(z_{k+1})}{ v'_{n, k+1}(z_{k+1}) } =- \frac{1}{   v'_{n, k+1}(z_{k+1}) } \,  \left(  \frac{  v'_{n,k}(u)}{  v_{n,k}(u)}\right)\bigg|_{u=z_{k+1}} \\
& =-  \frac{1}{z_{k+1}^2 \,  v'_{n, k+1}(z_{k+1}) } \, z_{k+1}^2 \left(  \frac{  v'_{n,k}(u)}{  v_{n,k}(u)}\right)'\bigg|_{u=z_{k+1}}
\label{expressionepsilonMod}
\end{split}
\end{equation}
and
$$
\varepsilon_{n,k}^{(2)} (\zeta) = \frac{ v_{n, k}' \left( z_k \right)- v_{n,k}'\left( z_{k+1} \right)}{    z_k-z_{k+1} } \, \frac{ 1}{  v_{n,k+1}'\left( z_{k+1} \right)  }.
$$
Combining \eqref{identityPhi1} and \eqref{identitiesderivativesPhi} we get that 
\begin{equation}
    \label{intermediate}
    \left(z_{k+1}-z_k\right) \left(1- \frac{\varepsilon_{n,k}^{(2)} (\zeta)}{n \zeta} \right) = -\frac{1}{n \zeta} -\frac{\varepsilon_{n,k}^{(1)} (\zeta)}{n^2 \zeta}  -\frac{\delta_{n, k}(\zeta)}{n^2 \zeta^2}.
\end{equation}
To prove \eqref{boundDelta}, it suffices to show that $\varepsilon_{n,k}^{(1)} (\zeta)$, $\varepsilon_{n,k}^{(2)} (\zeta)$, and $\delta_{n,k} (\zeta)$ are uniformly bounded for sufficiently small $|\zeta|$ (or equivalently sufficiently large $z_k$ and $z_{k+1}$). 

By \eqref{identitiesderivativesPhi} and Proposition~\ref{prob:boundsCauchy}, $\varepsilon_{n,k}^{(1)}(\zeta)$ is analytic for $|z_{k+1}|\ge R$, and functions
$$
u\frac{ v'_{n,k}(u)}{  v_{n,k}(u)} \quad \text{and} \quad \ u^2\,   v'_{n,k+1}(u) 
$$
are analytic and non-vanishing in $|u|\ge R$. By \eqref{boundsCauchy1}--\eqref{boundsCauchy2}, for $|u|\ge R$,
$$
 \left| u\, \frac{ v'_{n,k}(u)}{  v_{n,k}(u)} \right|= \left|  \frac{u^2    v'_{n,k}(u)}{ u\,  v_{n,k}(u)} \right|\le 18,
$$
and in consequence,
$$
\left| \frac{  v'_{n,k}(u)}{  v_{n,k}(u)} \right|\le \frac{18}{R},  \qquad |u|=R.
$$
By the Cauchy integral formula,
\begin{equation}
    \label{estimate1onPhi}
\left| u^2  \left( \frac{ v'_{n,k}(u)}{  v_{n,k}(u)} \right)'\right|\le 18 (R+1)^2,  \qquad |u|=R+1.
\end{equation}
On the other hand, by \eqref{boundsCauchy2},
\begin{equation}
    \label{boundPhi}
\left| \frac{1}{ u^2\,    v'_{n, k+1}(u) } \right| \le \frac{1}{1-(k+1)/n}\frac{(1+r/R)^4}{1-2r/R}\le \frac{4(3+2 \sqrt{2})}{1-(k+1)/n}, \quad |u|\ge R\ge (\sqrt{2}+1)r. 
\end{equation}
Using these estimates in \eqref{expressionepsilonMod}, we conclude that $\left| \varepsilon_{n,k}^{(1)}(\zeta)\right|$ is uniformly bounded for $|z_{k+1}|\ge R+1$ and $k/n\le T<1$. 

As for $\varepsilon_{n,k}^{(2)} (\zeta)$, we notice that for $z_k, z_{k+1}\in \Omega_R$,
$$
\left|  \frac{ v_{n, k}' \left( z_k \right)- v_{n,k}'\left( z_{k+1} \right)}{    z_k-z_{k+1} }\right|  \leq \max_{s\in \Omega_R} \left| v_{n,k}''\left( s \right)\right| .
$$
Appealing again to \eqref{boundsCauchy1} and Cauchy's formula, we get the uniform boundedness of $\varepsilon_{n,k}^{(2)} (\zeta)$ for  $z_k, z_{k+1}\in \Omega_{R+1}$.

Finally, by \eqref{taylorEx}, $\delta_{n, k}(\zeta)$ in \eqref{identityPhi1} satisfies   
\begin{equation}
    \label{eq:13}
\left|\delta_{n, k}(\zeta)\right| \leq \frac{1}{2 }   \max _{\zeta \in \D_\rho} \frac{ \varphi_{n, k+1}'' \left(\zeta\right)}{ \left( \varphi_{n, k}' \left(\zeta\right)\right)^2 },\quad \zeta\in \D_\rho.
\end{equation}
We have
\begin{align*}
\frac{ \varphi_{n, k+1}'' \left(\zeta\right)}{ \left( \varphi_{n, k}' \left(\zeta\right)\right)^2 } 
&= - \frac{   \,v_{n, k+1}''  \left(z_{k+1} \right)}{ \left(   v_{n, k+1}' \left(z_{k+1}\right)\right)^3 } \left(  v_{n, k}' \left(z_k\right)\right)^2 
= - \frac{   v''_{n, k+1} \left(z_{k+1} \right)}{    v'_{n, k+1} \left(z_{k+1}\right)  } \left( \frac{  v'_{n, k} \left(z_k\right)}{  v'_{n, k+1} \left(z_{k+1} \right)}\right)^2 
\\
& = - \frac{   v''_{n, k+1} \left(z_{k+1} \right)}{   v'_{n, k+1} \left(z_{k+1}\right)  } \left(  1+ \frac{\varepsilon_{n,k}(\zeta)}{n}\right)^2.
\end{align*}
Reasoning as above, by \eqref{boundsCauchy1},
$$
\left|   v''_{n, k+1} \left(u \right)\right| \le   \frac{4\left( 1-(k+1)/n\right)}{R}, \quad |u|=R+1,
$$
while, by \eqref{boundPhi},
$$
\left| \frac{1}{  v'_{n, k+1}(u) } \right| \le   \frac{4(3+2 \sqrt{2}) R^2}{1-(k+1)/n}, \quad |u|\ge R\ge (\sqrt{2}+1)r,
$$
showing that $\delta_{n,k}$ in \eqref{identityPhi1} is uniformly bounded also.
\end{proof}

We are one step away from proving Theorem~\ref{mainthm2}. Since $u(z,t)$ is the solution of \eqref{pde} with the initial condition \eqref{pdeInitial}, it satisfies equation \eqref{charact2}, which implies that 
$$
\left(u_{n, k+1}^{-1}-u_{n, k}^{-1}\right)(\zeta)=-\frac{1}{n \zeta}.
$$
Combining it with \eqref{eq:12} we get that
$$
\left(v_{n,k+1}^{-1}\left(\zeta\right) - u_{n,k+1}^{-1}\left(\zeta\right) \right) - \left(v_{n, k}^{-1}(\zeta) - u_{n,k}^{-1}\left(\zeta\right)\right) = \Delta_{n, k}(\zeta)=\mathcal O \left(\frac{1}{n^2} \right), \quad |\zeta|=\rho',
$$
uniformly over $n, k \in \N$ such that $k / n<T$. Since $u_{n, 0}^{-1}=v_{n, 0}^{-1}$, we conclude that 
\begin{equation}
    \label{proximityInverses}
u_{n, k}^{-1}(\zeta)-v_{n, k}^{-1}(\zeta)=\mathcal O\left(\frac{1}{n}\right)
\end{equation}
locally uniformly in $\mathbb D_\rho$.

In order to complete the proof, it remains to show that \eqref{proximityInverses} implies proximity of the direct functions, that is, 
\begin{equation}
    \label{proximityDirect}
u_{n, k} (z)-v_{n, k} (z)=\mathcal O\left(\frac{1}{n}\right)
\end{equation}
locally uniformly in $\Omega_R$ for some $R>r$. For that, it is more convenient to work with the functions
$$
\widehat u_{n,k}(z)\isdef u_{n,k}\left(1/z \right), \quad \widehat v_{n,k}(z,t)\isdef v_{n,k}\left(1/z,t\right).
$$
holomorphic in $\overline{\D_{1/R}}$ and non-vanishing in $\overline{\D_{1/R}} \setminus \{0\}$. Also, 
$$
\widehat u_{n, k}' (0)=\widehat v_{n, k}' (0)=1-\frac{k}{n}\ge 1-T>0, 
$$
and both $\widehat u_{n, k}$ and $\widehat v_{n, k}$ are uniformly bounded in $\overline{\D_{1/R}}$: there exists $M>0$, depending on $T$ and $R$ only, such that
$$
\left| \widehat u_{n, k}(z)\right| \leq M, \quad \left| \widehat v_{n, k}(z)\right| \leq M, \qquad z\in \overline{\D_{1/R}} ,
$$
for every pair $(k, n)\in \N\cup \{0\} \times \N$, with $k/n\in [0,T]$. The assertion \eqref{proximityDirect} is a consequence of the general facts gathered in Propositions~\ref{LemmaAux1}--\ref{LemmaAux3} that, for the sake of readability, we deferred to Appendix~\ref{appendix}. This concludes the proof of Theorem~\ref{mainthm2}.

\begin{proof}[Proof of Corollary~\ref{thm:Steiner}]
    By \eqref{defRtransform}--\eqref{propRtransform},
    $$ 
	 \mathcal R_{\mu_0^{\boxplus 1/(1-t)}} (w) = \frac{1}{1-t}\mathcal R_{\mu_0 } (w)  = \frac{1}{1-t} \left[ \mathcal C^{\mu_0}\right]^{-1}(w)-\frac{1}{(1-t) w}.
 $$
  where $\left[ \mathcal C^\mu\right]^{-1}$ is the functional inverse of $ \mathcal C^\mu$.	Taking into account \eqref{charact2}, we have that
  \begin{align*}
       \mathcal R_{\mu_0^{\boxplus 1/(1-t)}} (w) &=  \frac{1}{1-t} u^{-1} (w, 0) -\frac{1}{(1-t) w} =  \frac{1}{1-t} \left( u^{-1} (w, t) +\frac{t}{w}\right) -\frac{1}{(1-t) w} \\ &= \frac{1}{1-t}   u^{-1} (w, t)-\frac{1}{  w} ,
  \end{align*}
so that
$$
\left[ \mathcal C^{\mu_0^{\boxplus 1/(1-t)}}\right]^{-1}(w) = \frac{1}{1-t}   u^{-1} (w, t) ,
$$
or, equivalently,
\begin{equation}
    \label{identityzerosscaled}
\mathcal C^{\mu_0^{\boxplus 1/(1-t)}}\left(z \right)= \mathcal C^{\mu_t}((1-t) z)=u((1-t)z,t) .
\end{equation}
On the other hand, for  
$$   
\widehat \sigma_{n,k} \isdef \frac{1}{n-k} \, \chi\left(\widehat  Q_{n,k} \right)  
$$
we have that
$$
\mathcal C^{\widehat \sigma_{n,k} }\left(z \right)=  \frac{1}{n-k} \frac{ \partial_z\,  \widehat Q_{n,k}(z)}{\widehat Q_{n,k}(z)} = (1-t) \frac{1}{n-k} \frac{  Q'_{n,k}(z)}{ Q_{n,k}(z)} \left((1-t)z \right) = (1-t) \frac{n}{n-k} \mathcal C^{  \sigma_{n,k} }\left((1-t)z \right),
$$
with $ \sigma_{n,k}$ defined in \eqref{eq:defSigmaNK}. Hence, for $k/n\to t$, by Theorem~\ref{mainthm2},  
$$
\lim_{n }\mathcal C^{\widehat \sigma_{n,k} }\left(z \right)= \mathcal C^{  \mu_t }\left((1-t)z \right),
$$
so that by \eqref{identityzerosscaled},
   $$
\lim_{n }\mathcal C^{\widehat \sigma_{n,k} }\left(z \right)= \mathcal C^{\mu_0^{\boxplus 1/(1-t)}}\left(z \right).
$$
For measures on the real line, this is equivalent to \eqref{limitFreeConv}.
\end{proof}

\section{Some examples} \label{sec:examples}

We finish our exposition by presenting a few simple examples, starting with some paradigmatic zero distributions on the real line. 

It is well known that under very mild assumptions on the orthogonality measure on the interval $[-1,1]$ (such as having a strictly positive Radon-Nikodym derivative with respect to the Lebesgue measure almost everywhere on $[-1,1]$, see, e.g. \cite{StTo}), the limit zero distribution of the corresponding orthogonal polynomials is the the equilibrium measure (or arc-sine distribution) of the interval $[-1,1]$, that is,
$$
d\mu_0(x)=\frac{1}{\pi \sqrt{1-x^2}}\, dx, \quad x\in (-1,1),
$$
for which 
$$
u(z,0)=\mathcal C^{\mu_0}(z)=\frac{1}{\sqrt{z^2-1}} =\frac{1}{z} + \mathcal O\left(\frac{1}{z^2} \right), \quad z\in \C\setminus [-1,1],
$$
and consequently,
$$
u^{-1} (w, 0)=\frac{\sqrt{w^2+1}}{w}.
$$
By \eqref{charact2}, solving the equation
$$
z= \frac{\sqrt{w^2+1}-t}{w}
$$
we get that
$$
w=u(z,t) = \frac{-tz + \sqrt{z^2-(1-t^2)}}{z^2-1} . 
$$
Since the residues of the right hand side at $\pm 1$ are $0$, we conclude that $\mu_t$ is absolutely continuous and
	\begin{equation}\label{eqArcsine}
d\mu_t(x) = \frac{1}{\pi} \frac{\sqrt{1-t^2-x^2}}{1-x^2}\, dx, \qquad x\in [-\sqrt{1-t^2}, \sqrt{1-t^2}],
	\end{equation}
where we have used the Stieltjes--Perron inversion formula. 

\medskip

Another notable case is the semicircle distribution on the interval $[-1,1]$,
	$$
	d\mu_0(x)=\frac{2}{\pi }\, \sqrt{1-x^2}\, dx, \quad x\in (-1,1), 
	$$
corresponding, for instance, to the zero limit distribution of the appropriately rescaled Hermite polynomials. Since
	$$
	u(z,0)=2 \left( z-\sqrt{z^2-1}\right)  =\frac{1}{z} + \mathcal O\left(\frac{1}{z^2} \right), \quad z\in \C\setminus [-1,1],
	$$
	and
	$$
	u^{-1} (w, 0)=\frac{ w^2+4}{4w},
	$$
solving the equation
	$$
	z= \frac{ w^2+4-4t}{4w},
	$$
	we get that
	$$
w=	u(z,t) = 2 \left( z-\sqrt{z^2-(1-t)}\right) .
	$$
	Again, $\mu_t$ is absolutely continuous probability measure  given by
	\begin{equation}\label{eqSemicircle}
		d\mu_t(x) = \frac{2}{\pi }  \sqrt{1-t-x^2} \, dx, \quad x\in [-\sqrt{1-t}, \sqrt{1-t}],
	\end{equation}
the rescaled semicircle distribution for all values of $t\in [0,1)$.

\medskip

The case of the Marchenko-Pastur distribution on $[0,4]$, given by
$$
d\mu_0(x)=\frac{1}{2\pi} \, \frac{\sqrt{x(4-x)}}{x}\, dx,
$$
that describes the asymptotic behavior of singular values of large random matrices (or equivalently, the asymptotic zero distribution of rescaled Laguerre polynomials) is analyzed in a similar way.

\medskip

On the other hand, a natural case to consider is of the normalized uniform distribution on $[-1,1]$. It turns out it is remarkably challenging; indeed, in this case
	$$
	u(z,0)=\frac{1}{2}\log \left( \frac{z+1}{z-1}\right) , \quad z\in \C\setminus [-1,1],
	$$
	and
	$$
	u^{-1} (w, 0)=\frac{ v+1}{v-1}, \quad v=e^{2w}.
	$$
However, the equation\begin{equation} \label{lebesgue}
	z= \frac{ v+1}{v-1} - \frac{t}{\log v}, \quad v=e^{2w}
	\end{equation}
is highly transcendental and does not have an explicit solution. We can use this representation to find the support of the measure $\mu_t$, which again appears to be a formidable task.

\medskip

Let us end with the case where $\mu_0$ is discrete and is the normalized zero-counting measure of a given polynomial $P$, of degree $m$, with real or complex zeros. This is the case, for instance, addressed at the beginning, when $Q_n$ are generated by a Rodrigues-type formula. As mentioned, this situation has been studied before in \cite{ShapiroRodrigues} using different techniques. 

Now
 $$
 u(z, 0)=\frac{1}{m}\, \frac{d}{d z} \log P(z),
 $$
and equation  \eqref{Shapiro} reduces to\footnote{\, This equation appears in \cite[formula (1.6)]{ShapiroRodrigues} in the situation that is a particular case of a discrete measure $\mu_0 $.}
 $$
m P(z+t v)-v P^{\prime}(z+t v)=0, \quad v=\frac{1}{w} .
 $$
 Using that
 $$
 P(z+s)=\sum_{j=0}^m \frac{P^{(j)}(z)}{j !} s^j
 $$
 this yields an algebraic equation for $w=u(z,t)$:
 \begin{equation}\label{eqShapiro}
 \sum_{j=0}^m  \left(m-\frac{j}{t}\right) \frac{P^{(j)}(z)}{j !} \, t^j w^{m-j}=0 
 \end{equation}
 also derived in \cite{ShapiroRodrigues}.

For instance, in our first example \eqref{eq:Rodrigues}, when $P(z)=z^2-1$ and $m=2$, 
	$$
	u(z,0)= \frac{z }{z^2-1}, 
	$$
	and \eqref{eqShapiro} yields the quadratic equation for $w=u(z,t)$:
	$$
	(z^2-1)  w^2- (1-2t) z w -   (1-t) t=0. 
	$$
Thus,
	$$
	u(z,t) = \frac{(1-2t) z -\sqrt{4 (t-1) t+z^2} }{2 \left(z^2-1\right)},
	$$
and we get the expressions \eqref{example1a}--\eqref{example1b} for measure $\mu_t$. 

\begin{remark} \label{remarkBound}
Ravichandran \cite{MR4130852} (see also \cite{MR4279350}) proved that for $t\ge 1/2$, the ratio of the diameters of the sets of zeros of polynomials $Q_{n,k}(x) $, defined in \eqref{def:QN}, and $Q_n$, for $k=[nt]$ is bounded above by $2\sqrt{t(1-t)}$. Obviously, both \eqref{eqArcsine} and \eqref{eqSemicircle} are in agreement with this bound, and formulas \eqref{example1a}--\eqref{example1b} show that the assumption that $t\ge 1/2$ cannot be omitted. 
\end{remark}

On the other hand, in Kalyagin's example \eqref{Kalyaginexample}, when $P(z)=z(z^2-1)$ and $m=3$, 
\eqref{eqShapiro} yields the following cubic equation for $w=u(z,t)$:
$$
3  z \left(z^2-1\right) w^3+(3 t-1) \left(3 z^2-1\right)w^2+3 t (3 t-2) z w  -3t^2 (1-t) =0.
$$
For $t\in (0,1)$, the end-points of the support of $\mu_t$ can be found from the branch points of $w$, which are solutions of the bi-quadratic equation
$$
9   z^4 + 3 \left(9 t^2-6
   t-2\right) z^2-(1-t) (3 t-1)^3=0,
   $$
see \cite{kaliaguine:1981}.

Finally, if $P$ has complex roots, such as $P(z)=z^3-1$,  \eqref{eqShapiro} yields again a cubic equation for $w=u(z,t)$:
 $$
 (z^3-1) w^3 - (1-3t) z^2 w^2 - (2-3t) t z w -(1-t) t^2=0,
 $$
whose branch points are solutions of the equation
 $$
 27 t^3-54 t^2+27 t-4 z^3=0,
 $$
 that is, the three cube roots
 $$
 z=    3 \sqrt[3]{\frac{t (1-t)^2}{4}}. 
$$ 
  
\appendix

\section{Auxiliary results} \label{appendix}

Denote by $\mathcal H(R, M, c)$ the class of functions $f$ satisfying the following conditions:
\begin{enumerate}[(i)]
    \item $f$ is analytic in the closed disc $\overline{ \D_R}$;
\item $|f(z)| \leq M$ for $  z \in \overline{ \D_R}$;
\item $f'(0) =c > 0$.
\end{enumerate}

\begin{prop}
    \label{LemmaAux1}
For $R, M, c>0$, let 
\begin{equation}\label{r1}
r_1 \isdef \min\left\{\frac R2, \, \frac{cR^2}{4 M}\right\}.
\end{equation}
Then for $r \leq r_1$, any function $f\in \mathcal H(R, M, c)$ is univalent in $\overline{\D_r}$, 
$f(\overline{\D_r})$ covers the disc $\overline{\D_\rho}$ with $\rho = cr/2$.  

Consequently, $f^{-1}$ is univalent in $\overline{\D_{\rho}}$ and $f^{-1}(\overline{D_{\rho}})\subset \D_r$.
\end{prop}
\begin{proof}
Let $f(z) = cz + h(z)\in \mathcal H(R, M, c)$, $h(z) = \sum_{ n =2}^\infty c_n z^n$. Since $h$ is holomorphic in the closed disc $\overline{ \D_R}$, the Cauchy unequalities imply that $|c_n| \leq  M / R^n$, $n =2, 3, \dots$, so that for any $r < R$, 
$$ 
|h(z)| \leq M \sum_{ n =2}^\infty (r/R)^n =  M \frac {(r/R)^2}{1 - (r/R)}, \quad z \in \D_r.
$$
If $r\leq   R/2$ then $1 - r/R \geq 1/2$, and 
$$
|h(z)| \leq 2 M (r/R)^2, \quad  z \in \D_r.
$$
Also,
$$
r \leq  \frac {cR^2}{4 M} \quad \Leftrightarrow \quad  cr \geq 4  M \left(\frac rR\right)^2 .
$$
This means that if we assume  $r \leq r_1$, it implies that for $z \in C_r\isdef \partial \D_r=\{z\in \C:\, |z|=r\}$, $|cz| \geq 2 |h(z)|>0$ . This means that the image of the circle $C_r$ by $f$ lies in the annulus $ cr/2=\rho  \leq |z| \leq 3\rho$ and, by Rouch\'e's theorem, its index with respect to the origin is one. The index of $f(C_r)$ with respect to any point $f(a)\notin f(C_r)$ is constant in any connected component of the complement to $f(C_r)$, and is equal to the multiplicity of $f(a)$ in $D_r$. We conclude that this index is $1$ for any $a\in \D_r$, concluding that $f$ is univalent in $\D_r$. 
\end{proof}

\begin{prop}
	\label{LemmaAux2} 
	Let   $f \in \mathcal H(R, M, c)$ and
\begin{equation}\label{1}
r_0 \isdef  \min\left\{\frac R2, \, \frac{cR^2}{16 M}\right\}\le r_1
\end{equation}
where $r_1$ was defined in \eqref{r1}. 

Then for any $r \leq r_0$ we have $|f'(z)| \geq c/2 $ for $z \in \D_r$.
\end{prop}
\begin{proof}
Since $f'(z) = c + h'(z)$, $h(z) = \sum_{ n =2}^\infty c_n z^n$, the Cauchy inequalities imply again that  for $z \in \D_r$,
\begin{align*}
|h'(z)| & \leq \left|  \sum_{ n =2}^\infty n c_n z^{n-1}\right| \leq M  \sum_{ n =2}^\infty n \frac{r^{n-1}}{R^n} =
\frac{Mr}{R^2} \sum_{ n = 1}^\infty (n+1)\left(\frac rR\right)^{n -1} \\
& \leq \frac{2Mr}{R^2}\sum_{ n = 1}^\infty n\left(\frac rR\right)^{n -1} 
= \frac{2Mr}{R^2} \frac 1{(1-r/R)^2} = \frac {2Mr}{(R-r)^2}.
\end{align*}
If $r \leq R/2$, then $|h'(z)| \leq 2 Mr/r^2=2M/r$ for $ z \in \D_r$. On the other hand, if additionally $r\le cR^2/(16 M) $ then  $c \geq 16 Mr/ R^2\ge 4M/r$, which implies that for $ z \in \D_r$,
$$ 
|f'(z)| \geq |c + h'(z)| \geq c - |h'(z)| \geq c - 2M/r \geq c - 8Mr/R^2 \geq c/2.
$$
\end{proof}

\begin{prop}
    \label{LemmaAux3}
    Let $f,g \in \mathcal H(R, M, c)$ and for some $\epsilon>0$, $|f(z)-g(z)|<\epsilon$ for $z\in \D_R$. Then there exist positive $\rho$ and $C$, depending only on $R$, $M$, and $c$, such that the inverse functions $f^{-1}$ and $g^{-1}$ exist in $D_\rho$, and
\begin{equation} \label{finalinequ}
 m(\rho)\isdef \max_{\zeta\in \overline{\D_\rho}} \left| f^{-1}(\zeta) - g^{-1}(\zeta)\right| \le C\epsilon. 
\end{equation}
\end{prop}
\begin{proof}
Fix $r_0 $ defined in \eqref{1} and consider $r \in (0, r_0) $, as long as the corresponding $\rho = cr/2$.  By Proposition~\ref{LemmaAux1}, for any such $r$ and $\rho$, functions $f^{-1}(\zeta)$, $g^{-1}(\zeta)$ exist, are  univalent in $\overline{\D_{\rho}}$ and map $\overline{D_{\rho}}$ into $  \D_r$. Thus, for any $\rho \leq c/2$ and $\zeta \in \D_\rho$,
\begin{equation}\label{3}
	\begin{split}
|f^{-1}-g^{-1}|(\zeta)  & = \left|\int_0^\zeta \left[ (f^{-1})'(w) - (g^{-1})'(w)\right] dw \right| =  
\int_0^\zeta  \left |\frac 1 { f'(f^{-1}(w))} - \frac 1 { g'(g^{-1}(w))}\right | |dw | 
\\
& \le  \rho \max_{w \in \D_{\rho}}\left |\frac{ g'(g^{-1}(w) )-  f'(f^{-1}(w))} { g'(g^{-1}(w))  f'(f^{-1}(w))}  \right| \leq
    \frac {4 \rho }{c^2} \max_{w \in \D_{\rho}}\left | g'(g^{-1}(w)) - f'(f^{-1}(w)) \right| ,
    	\end{split}
\end{equation}
where we have used Proposition~\ref{LemmaAux2}. For $w \in \D_\rho$,
$$
g'(g^{-1}(w) )- f'(f^{-1}(w) )= \left( g'(g^{-1}(w)) - f'(g^{-1}(w) )\right) + \left ( f'(g^{-1}(w)) - f'(f^{-1}(w)) \right) 
$$
and thus,    
\begin{equation}\label{4}
\left | g'(g^{-1}(w) )- f'(f^{-1}(w)) \right| \leq \left | (g' - f') (g^{-1}(w)) \right| + \left | f'(g^{-1}(w)) - f'(f^{-1}(w)) \right| .
\end{equation}

Observe that by the Cauchy integral formula, for  $f\in \mathcal H(R, M, c)$ and $z\in \D_r$, with any $r < R$, we have  $|f'(z)| \leq M/(R - r)$ and $|2 f''(z)| \leq M/(R - r)^2$. 
Thus, since for $w \in \D_\rho$, $ z=g^{-1}(w) \in \D_r$, we have
$$
 \left | g'(g^{-1}(w) )- f'(g^{-1}(w)) \right| = \left | (g' - f')(z) \right| \leq \frac \epsilon{R-r} \leq 2 \epsilon
$$
and 
$$
\left | f'(g^{-1}(w)) - f'(f^{-1}(w)) \right| \leq \max_{z \in \D_r} |f''(z)|  \max_{w \in \D_\rho}\left| g^{-1}(w) - f^{-1}(w) \right|\le \frac{M}{2 (R-r)^2 }\, m(\rho)\le \frac{2M}{ R ^2 }\, m(\rho).
$$
Combining these bounds in \eqref{4} we obtain from \eqref{3}
$$
m(\rho) \leq  \frac {4 \rho }{c^2}  \left( 2\epsilon +\frac{2M}{ R ^2 }\, m(\rho) \right).
$$
Thus, selecting for instance $\rho = c^2 R^2/ (16 M)$ we obtain \eqref{finalinequ} with $C = M/R^2$. 
\end{proof}

\section*{Acknowledgments}

The first author was partially supported by Simons Foundation Collaboration Grants for Mathematicians (grant MPS-TSM-00710499).
He also acknowledges the support of the project PID2021-124472NB-I00, funded by MCIN/AEI/10.13039/501100011033 and by ``ERDF A way of making Europe'', as well as the support of Junta de Andaluc\'{\i}a (research group FQM-229 and Instituto Interuniversitario Carlos I de F\'{\i}sica Te\'orica y Computacional). 

The authors also thank Brian Hall and Edward Saff for interesting discussions, as well as the anonymous referee for very useful observations.

\def\cprime{$'$}


\begin{thebibliography}{10}
	
	\bibitem{angst2023sure}
	J.~Angst, D.~Malicet, and G.~Poly.
	\newblock Almost sure behavior of the critical points of random polynomials,
	2023.
	
	\bibitem{Aptekarev:97}
	A.~I. Aptekarev, F.~Marcell\'an, and I.~A. Rocha.
	\newblock Semiclassical multiple orthogonal polynomials and the properties of
	{J}acobi-{B}essel polynomials.
	\newblock {\em J.\ Approx.\ Theory}, 90:117--146, 1997.
	
	\bibitem{MR2262808}
	A.~I. Aptekarev and Y.~G. Rykov.
	\newblock On the variational representation of solutions to some quasilinear
	equations and systems of hyperbolic type on the basis of potential theory.
	\newblock {\em Russ. J. Math. Phys.}, 13(1):4--12, 2006.
	
	\bibitem{Arendt2023}
	W.~Arendt and K.~Urban.
	\newblock {\em Partial Differential Equations. {An} Introduction to Analytical
		and Numerical Methods}, volume 294 of {\em Graduate Texts in Mathematics}.
	\newblock Springer Verlag, 2023.
	
	\bibitem{Arizmendi21}
	O.~Arizmendi, J.~Garza-Vargas, and D.~Perales.
	\newblock Finite free cumulants: Multiplicative convolutions, genus expansion
	and infinitesimal distributions.
	\newblock {\em Transactions of the American Mathematical Society},
	376(06):4383--4420, 2023.
	
	\bibitem{ShapiroRodrigues}
	R.~B{\o}gvad, C.~H{\"a}gg, and B.~Shapiro.
	\newblock {Rodrigues'} descendants of a polynomial and {Boutroux} curves.
	\newblock \textit{Constr. Approx.}, 59:737--798, 2024.
	
	\bibitem{MR4474893}
	S.-S. Byun, J.~Lee, and T.~R. Reddy.
	\newblock Zeros of random polynomials and their higher derivatives.
	\newblock {\em Trans. Amer. Math. Soc.}, 375(9):6311--6335, 2022.
	
	\bibitem{campbell2023fractional}
	A.~Campbell, S.~O'Rourke, and D.~Renfrew.
	\newblock The fractional free convolution of {$R$}-diagonal operators and
	random polynomials under repeated differentiation, 2023.
	\newblock Preprint {arXiv:2307.11935}.
	
	\bibitem{MR3573689}
	A.~del Monaco and S.~Schlei\ss inger.
	\newblock Multiple {SLE} and the complex {B}urgers equation.
	\newblock {\em Math. Nachr.}, 289(16):2007--2018, 2016.
	
	\bibitem{MR0708494}
	P.~L. Duren.
	\newblock {\em Univalent functions}, volume 259 of {\em Grundlehren der
		mathematischen Wissenschaften [Fundamental Principles of Mathematical
		Sciences]}.
	\newblock Springer-Verlag, New York, 1983.
	
	\bibitem{MR1625845}
	L.~C. Evans.
	\newblock {\em Partial differential equations}, volume~19 of {\em Graduate
		Studies in Mathematics}.
	\newblock American Mathematical Society, Providence, RI, 1998.
	
	\bibitem{galligo2022modeling}
	A.~Galligo.
	\newblock Modeling complex root motion of real random polynomials under
	differentiation, 2022.
	\newblock Preprint {arXiv:2205.08747}.
	
	\bibitem{kabluchko2023fractional}
	B.~C. Hall, C.-W. Ho, J.~Jalowy, and Z.~Kabluchko.
	\newblock Roots of polynomials under repeated differentiation and repeated
	applications of fractional differential operators.
	\newblock Preprint {arXiv:2312.14883}, 2023.
	
	\bibitem{kabluchko2023heat}
	B.~C. Hall, C.-W. Ho, J.~Jalowy, and Z.~Kabluchko.
	\newblock Zeros of random polynomials undergoing the heat flow.
	\newblock Preprint {arXiv:2308.11685}, 2023.
	
	\bibitem{HoskinsKabluchko21}
	J.~Hoskins and Z.~Kabluchko.
	\newblock Dynamics of zeroes under repeated differentiation.
	\newblock \textit{Experimental Mathematics,} 32(4):573--599, 2021.
	
	\bibitem{MR4447137}
	J.~G. Hoskins and S.~Steinerberger.
	\newblock A semicircle law for derivatives of random polynomials.
	\newblock {\em Int. Math. Res. Not. IMRN}, 2022(13):9784--9809, 2022.
	
	\bibitem{MR3773856}
	I.~Hotta and M.~Katori.
	\newblock Hydrodynamic limit of multiple {SLE}.
	\newblock {\em J. Stat. Phys.}, 171(1):166--188, 2018.
	
	\bibitem{MR4259446}
	I.~Hotta and S.~Schlei\ss inger.
	\newblock Limits of radial multiple {SLE} and a {B}urgers-{L}oewner
	differential equation.
	\newblock {\em J. Theoret. Probab.}, 34(2):755--783, 2021.
	
	\bibitem{kabluchko2022repeated}
	Z.~Kabluchko.
	\newblock Repeated differentiation and free unitary Poisson process.
	\newblock Preprint {arXiv:2112.14729}, 2021.
	
	\bibitem{kaliaguine:1981}
	V.~A. Kaliaguine.
	\newblock On a class of polynomials defined by two orthogonality relations.
	\newblock {\em Math. USSR Sbornik}, 38(4):563--580, 1981.
	
	\bibitem{MR1618739}
	S.~Kerov.
	\newblock Interlacing measures.
	\newblock In {\em Kirillov's seminar on representation theory}, volume 181 of
	{\em Amer. Math. Soc. Transl. Ser. 2}, pages 35--83. Amer. Math. Soc.,
	Providence, RI, 1998.
	
	\bibitem{MR4458083}
	A.~Kiselev and C.~Tan.
	\newblock The flow of polynomial roots under differentiation.
	\newblock {\em Ann. PDE}, 8(2):Paper No. 16, 69, 2022.
	
	\bibitem{MR0140674}
	T.~H. MacGregor.
	\newblock Functions whose derivative has a positive real part.
	\newblock {\em Trans. Amer. Math. Soc.}, 104:532--537, 1962.
	
	\bibitem{MR4408504}
	A.~W. Marcus, D.~A. Spielman, and N.~Srivastava.
	\newblock Finite free convolutions of polynomials.
	\newblock {\em Probab. Theory Related Fields}, 182(3-4):807--848, 2022.
	
	\bibitem{MR2309862}
	P.~A. Markowich.
	\newblock {\em Applied partial differential equations: a visual approach}.
	\newblock Springer, Berlin, 2007.
	\newblock With 1 CD-ROM (Windows, Macintosh and UNIX).
	
	\bibitem{MR3302630}
	A.~Mart\'{i}nez-Finkelshtein, R.~Orive, and E.~A. Rakhmanov.
	\newblock Phase transitions and equilibrium measures in random matrix models.
	\newblock {\em Comm. Math. Phys.}, 333(3):1109--1173, 2015.
	
	\bibitem{MR3585560}
	J.~A. Mingo and R.~Speicher.
	\newblock {\em Free probability and random matrices}, volume~35 of {\em Fields
		Institute Monographs}.
	\newblock Springer, New York; Fields Institute for Research in Mathematical
	Sciences, Toronto, ON, 2017.
	
	\bibitem{MR2266879}
	A.~Nica and R.~Speicher.
	\newblock {\em Lectures on the combinatorics of free probability}, volume 335
	of {\em London Mathematical Society Lecture Note Series}.
	\newblock Cambridge University Press, Cambridge, 2006.
	
	\bibitem{MR4242313}
	S.~O'Rourke and S.~Steinerberger.
	\newblock A nonlocal transport equation modeling complex roots of polynomials
	under differentiation.
	\newblock {\em Proc. Amer. Math. Soc.}, 149(4):1581--1592, 2021.
	
	\bibitem{MR4130852}
	M.~Ravichandran.
	\newblock Principal submatrices, restricted invertibility, and a quantitative
	{G}auss-{L}ucas theorem.
	\newblock {\em Int. Math. Res. Not. IMRN}, 2020(15):4809--4832, 2020.
	
	\bibitem{salih2015inviscid}
	A.~Salih.
	\newblock Inviscid {Burgers'} equation.
	\newblock {\em A university lecture, Indian Institute of Space Science and
		Technology}, pages 1--19, 2015.
		
	\bibitem{SaTo}
	E.~B.~Saff and V.~Totik.
	\newblock {\em Logarithmic potentials with external fields}, second edition.
	\newblock Springer Verlag, Switzerland, 2024.
	
	\bibitem{ShlyakhtenkoTao21}
	D.~Shlyakhtenko and T.~Tao.
	\newblock Fractional free convolution powers.
	\newblock {\em Indiana Univ. Math. J.}, 71(6):2551--2594, 2022.
	
	\bibitem{MR4582606}
	V.~N. Sorokin.
	\newblock A generalization of the discrete {R}odrigues formula for {M}eixner
	polynomials.
	\newblock {\em Mat. Sb.}, 213(11):79--101, 2022.
	
	\bibitem{MR4582563}
	V.~N. Sorokin.
	\newblock On polynomials defined by the discrete {R}odrigues formula.
	\newblock {\em Mat. Zametki}, 113(3):423--439, 2023.
	
	\bibitem{StTo}
	H.~Stahl and V.~Totik.
	\newblock {\em General orthogonal polynomials}, volume~43 of {\em Encyclopedia
		of Mathematics and its Applications}.
	\newblock Cambridge University Press, Cambridge, 1992.
	
	\bibitem{MR4011508}
	S.~Steinerberger.
	\newblock A nonlocal transport equation describing roots of polynomials under
	differentiation.
	\newblock {\em Proc. Amer. Math. Soc.}, 147(11):4733--4744, 2019.
	
	\bibitem{Steiner2021}
	S.~Steinerberger.
	\newblock Free convolution powers via roots of polynomials.
	\newblock {\em Exp. Math.}, 32(4):567--572, 2023.
	
	\bibitem{MR4279350}
	V.~Totik.
	\newblock Critical points of polynomials.
	\newblock {\em Acta Math. Hungar.}, 164(2):499--517, 2021.
	
	\bibitem{VANDENHEUVEL2023133686}
	D.~J. Van den Heuvel, C.~J. Lustri, J.~R. King, I.~W. Turner, and S.~W. McCue.
	\newblock Burgers' equation in the complex plane.
	\newblock {\em Physica D: Nonlinear Phenomena}, 448:133686, 2023.
	
	\bibitem{MR1699025}
	G.~B. Whitham.
	\newblock {\em Linear and nonlinear waves}.
	\newblock Pure and Applied Mathematics (New York). John Wiley \& Sons, Inc.,
	New York, 1999.
	\newblock Reprint of the 1974 original, A Wiley-Interscience Publication.
	
\end{thebibliography}
\end{document}